\documentclass{article}

\input{hat-partition.sty}
\input{hat-macros.sty}

\title{A construction of the hat tilings by a Markov partition}

\author{Sébastien Labbé\footnote{CNRS – Université de Montréal CRM-CNRS, Montréal, Canada, \mailto{sebastien.labbe@cnrs.fr}}
  ~and
  Peter Selinger\footnote{Dalhousie University, Halifax, Canada, \mailto{selinger@mathstat.dal.ca}}}

\date{}

\begin{document}

\maketitle

\begin{abstract}
  We present a simple construction of hat tilings. The construction
  can be carried out by superimposing a triangular grid on a specially
  colored image and reading off the orientation of the tiles. We show
  that our construction produces valid hat tilings, and conversely,
  in an appropriate sense that is made precise in the paper, that
  every valid hat tiling can be obtained in this way.
\end{abstract}

\section{Non-technical overview of the construction}

The purpose of this short paper is to present a simple construction of
hat tilings \cite{MR4770585}. To keep the construction accessible to
non-specialists, this introductory section is deliberately written
without technical jargon. More mathematical details are given in the
later parts of this paper.

The \emph{hat tile} is the following shape. Because it is composed of
8 identical kite-shaped pieces (shown in light gray outlines), it is
an example of a shape called a \emph{polykite}.
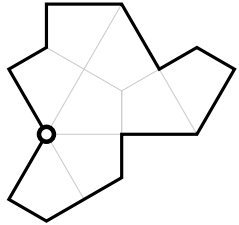
\[
\begin{tikzpicture}[scale=2,rotate=180]
  \draw[thin,black!20] (0,0) -- (-1,0) -- (-0.5,-0.866025) -- cycle;
  \draw[thin,black!20] (0,0) -- (-0.25,0.433013);
  \draw[thin,black!20] (-0.5,0) -- (-0.5,-0.288675) -- (-0.75,-0.433013);
  \draw[thin,black!20] (-0.5,0) -- (-0.5,-0.288675) -- (0,-0.577350);
  \draw[very thick] (0,0) \smithpath;
  \node[whitedot] at (0,0) {};
\end{tikzpicture}
\]
Smith, Myers, Kaplan and Goodman-Strauss {\cite{MR4770585}} discovered that the hat
tile has a remarkable property: while it is possible to tile the
infinite plane with copies of the hat tile and its mirror image (called anti-hat 
\cite{zbMATH08135691}), such tilings can never be periodic. 
A small section of a hat tiling is shown in Figure~\ref{fig:sample-tiling}(a).

\begin{figure}
\begin{center}
    \begin{tikzpicture}[
            scale=1.5,
        ]

        \begin{scope}
        \smith{1.2*1.000}{1.2*0}     {+1}
        \smith{1.2*0.500}{1.2*0.866025}{+2}
        \smith{1.2*-0.50}{1.2*0.866025}{+3}
        \smith{1.2*-1.00}{1.2*0}     {+4}
        \smith{1.2*-0.50}{1.2*-0.866025}{+5}
        \smith{1.2*0.500}{1.2*-0.866025}{+6}
        \end{scope}

        \begin{scope}[xshift=6cm]
        \smith{1.2*1.000}{1.2*0}     {-1}
        \smith{1.2*0.500}{1.2*0.866025}{-2}
        \smith{1.2*-0.50}{1.2*0.866025}{-3}
        \smith{1.2*-1.00}{1.2*0}     {-4}
        \smith{1.2*-0.50}{1.2*-0.866025}{-5}
        \smith{1.2*0.500}{1.2*-0.866025}{-6}
        \end{scope}

    \end{tikzpicture}
\end{center}
\caption{The 12 orientations of the hat tile. To each tile
  orientation, we assign a unique color and a label in the set
  $\s{+1,+2,+3,+4,+5,+6,-1,-2,-3,-4,-5,-6}$.}
\label{fig:anchored-tiles}
\end{figure}
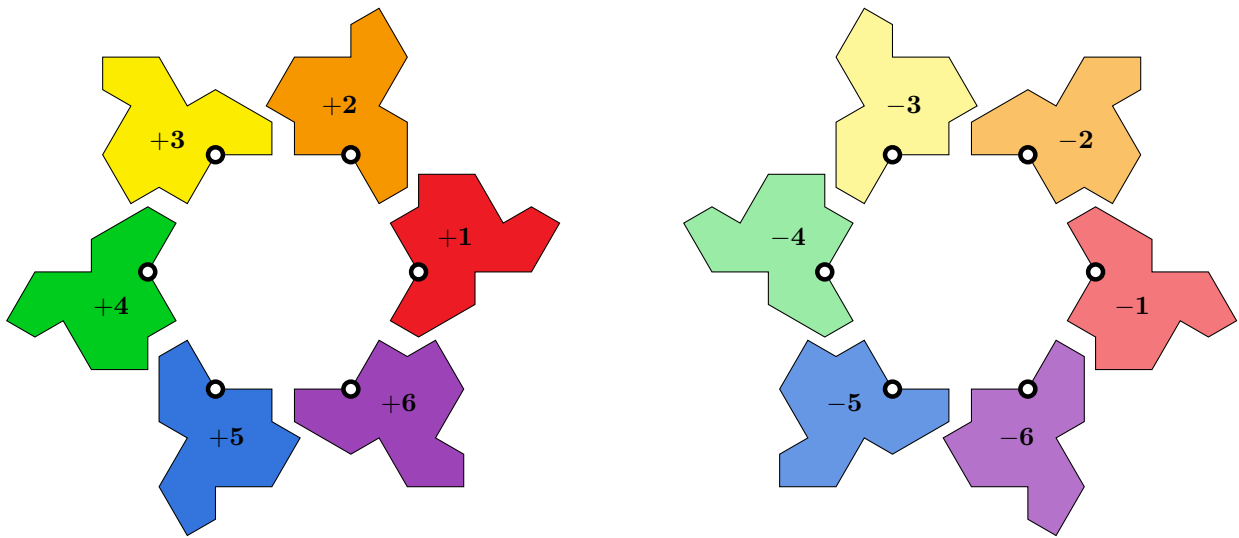
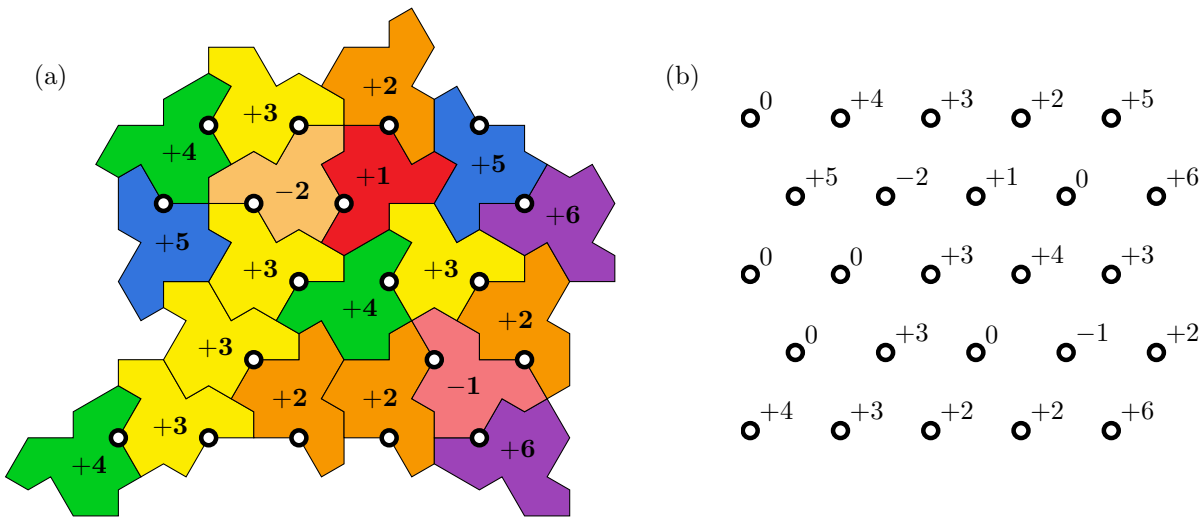
\begin{figure}
\begin{center}
\begin{tikzpicture}[
        scale=1.2
    ]
    \tikzstyle{every node}=[font=\normalsize]
    \node at (4.25,4) {(a)};
    \node at (11.25,4) {(b)};
\begin{scope}[xshift=5cm]
\smith{0}{0}{+4}  
\smith{1}{0}{+3}  
\smith{2}{0}{+2}  
\smith{3}{0}{+2}  
\smith{4}{0}{+6}  
\smith{1.5}{0.866025}{+3}  
\smith{3.5}{0.866025}{-1}  
\smith{4.5}{0.866025}{+2}  
\smith{2.0}{0.866025*2}{+3}  
\smith{3.0}{0.866025*2}{+4}  
\smith{4.0}{0.866025*2}{+3}  
\smith{0.5}{0.866025*3}{+5}  
\smith{1.5}{0.866025*3}{-2}  
\smith{2.5}{0.866025*3}{+1}  
\smith{4.5}{0.866025*3}{+6}  
\smith{1.0}{0.866025*4}{+4}  
\smith{2.0}{0.866025*4}{+3}  
\smith{3.0}{0.866025*4}{+2}  
\smith{4.0}{0.866025*4}{+5}  
\end{scope}
    \begin{scope}[xshift=13cm]
        \begin{scope} %
            \clip (-1.1,-.4) rectangle (3.7, 4.1);

            \foreach \a in {-3,-2,-1,0,1,2,3}
            \foreach \b in {0,1,2,3,4}{\
                \node[whitedotnone] at (\a+.5*\b,.866025*\b+.08) {};
            }
        \end{scope}

        \foreach \a/\b/\tilelabel in {
                -3/4/0 , 1/3/0   , 1/1/0, -2/2/0  , -1/2/0   , -1/1/0   ,
                            -2/4/+4 , -1/4/+3  , 0/4/+2 , 1/4/+5  ,
                            -2/3/+5 , -1/3/-2  , 0/3/+1 ,           2/3/+6,
                                                 0/2/+3 , 1/2/+4  , 2/2/+3  ,
                                                 0/1/+3           , 2/1/-1  , 3/1/+2,
                                      -1/0/+4  , 0/0/+3 , 1/0/+2  , 2/0/+2  , 3/0/+6}
        {
            \node[above right] at (\a+.5*\b,.866025*\b+.08) {$\tilelabel$};
        }
    \end{scope}
\end{tikzpicture}
\end{center}
\caption{(a) A portion of a hat tiling. The tiles are colored and
  labelled according to their orientations as in
  Figure~\ref{fig:anchored-tiles}.  The tiles' anchor points lie
  on a grid. (b) We have labelled each grid point with the orientation of
  the tile anchored there, or $0$ if no tile is anchored at the
  point. Note that the tiling is completely determined by the
  assignment of labels to grid points.}
\label{fig:sample-tiling}
\end{figure}

Although Smith et al.\@ gave an infinite family of related aperiodic
tiles, for simplicity, we focus here on the original ``hat''
polykite.  Our method for constructing a hat tiling is simple.
Consider the triangular grid of Figure~\ref{fig:triangular-grid},
whose grid points we have marked with small circles. We equip each hat
tile with an \emph{anchor} or \emph{control point} that is shown as a
small circle above.  The tiles will be placed on the triangular grid
such that their anchors coincide with grid points. Note that the kites
of each hat tile line up with kites on the triangular grid (shown as
light gray outlines in Figure~\ref{fig:triangular-grid}). Since each
hat tile's anchor is incident to four of its kites, forming an angle
of $4\cdot 60^{\circ}=240^{\circ}$, no two distinct tiles can be
anchored at the same position.  There are 12 possible orientations of
the hat tile, which are shown in Figure~\ref{fig:anchored-tiles}. For
convenience, we have colored each tile according to its orientation.

\begin{figure}
\begin{center}
\begin{tikzpicture}[
        scale=2
    ]
\begin{scope}
    \clip (-1.1,-.2) rectangle (3.1, 4*0.866025+.2);
    \begin{pgfonlayer}{labels}
      \clip (-1.1,-.2) rectangle (3.1, 4*0.866025+.2);
    \end{pgfonlayer}

    \foreach \a in {-4,-3,-2,-1,0,1,2,3,4}
    \foreach \b in {-1,0,1,2,3,4}{
      \begin{scope}[shift={(\a+.5*\b,0.866025*\b)}]
        \begin{pgfonlayer}{labels}
          \node[whitedotnone,fill=white] at (0,0) {};
        \end{pgfonlayer}
        \draw[gray] (0,0) -- (1,0);
        \draw[gray] (0,0) -- (0.5,0.866025);
        \draw[gray] (0,0) -- (-0.5,0.866025);
        \draw[thin,black!15] (0.5,-0.288675) -- (0.5,0.288675) -- (0,0.577350) -- (-0.5,0.288675);
      \end{scope}
    }
\end{scope}
\end{tikzpicture}
\end{center}
\caption{A triangular grid. The grid points are spaced 1 unit apart.}
\label{fig:triangular-grid}
\end{figure}
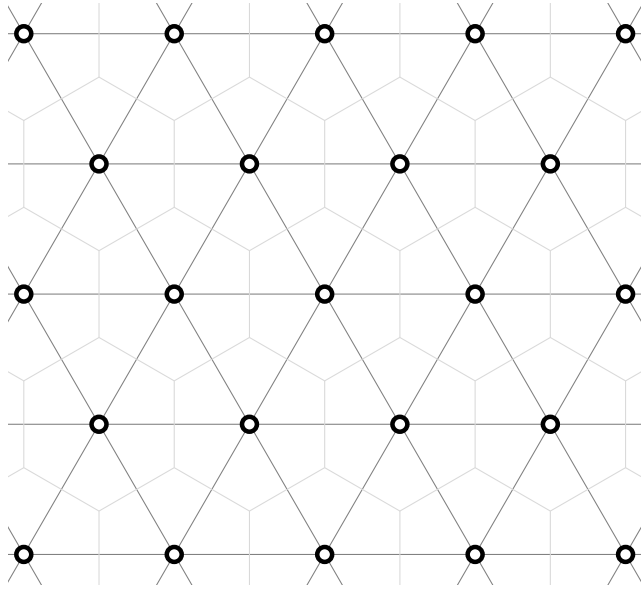
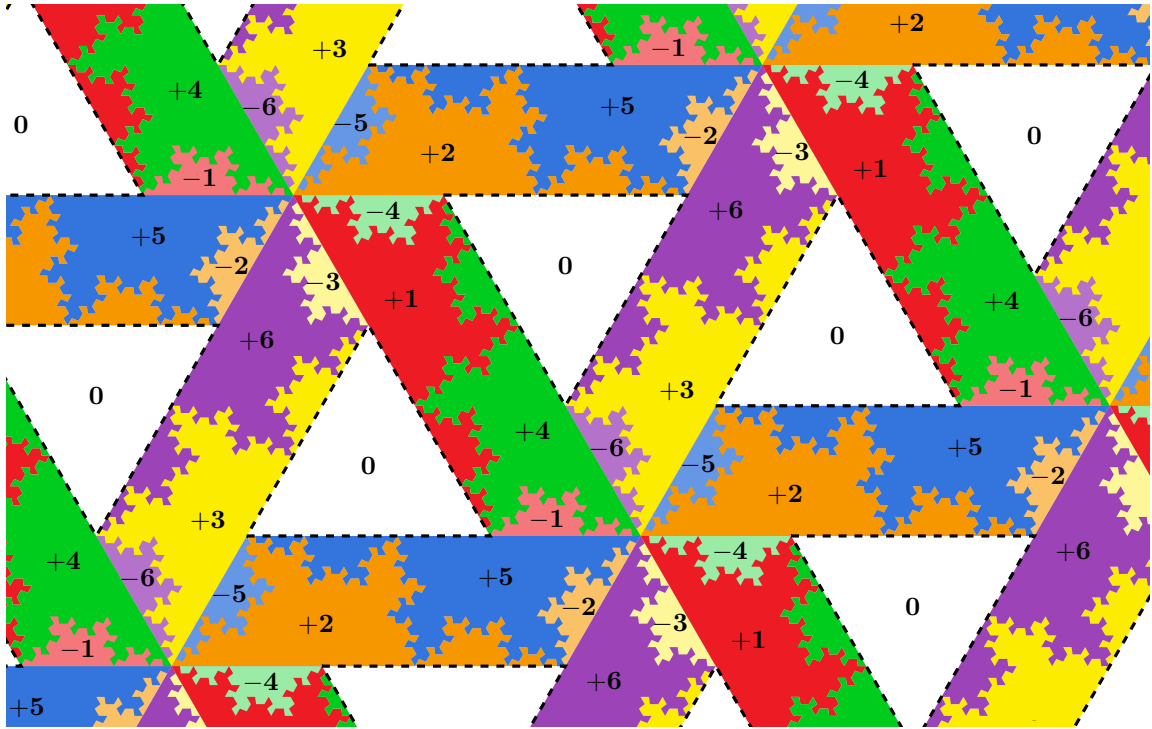
\begin{figure}
\begin{center}
\begin{tikzpicture}[
        scale=2
    ]
    \tikzstyle{every node}=[font=\normalsize]

    \clip (-1.1,-.4) rectangle (6.5, 4.4);
    \begin{pgfonlayer}{triangles}
      \clip (-1.1,-.4) rectangle (6.5, 4.4);
    \end{pgfonlayer}
    
    \foreach \a in {-1,0,1,2,3}
    \foreach \b in {-1,0,1,2}{
    \begin{scope}[xshift=\a*\Lxa cm + \b*\Lya cm,
                yshift=\a*\Lxb cm + \b*\Lyb cm]
            \FundamentalDomainAroundOrigin
    \end{scope}}
    \begin{scope}
    \clip (-1.1,-.36) rectangle (6.3, 4.4);
    \foreach \a in {-1,0,1,2,3}
    \foreach \b in {-1,0,1,2}{
    \begin{scope}[xshift=\a*\Lxa cm + \b*\Lya cm,
                yshift=\a*\Lxb cm + \b*\Lyb cm]
            \FundamentalDomainLabellingStandard
    \end{scope}}
    \end{scope}
\end{tikzpicture}
\end{center}
\caption{A Markov partition $\Plimit$ describing hat tilings with more hats
    than anti-hats. We have partitioned
  the plane into 13 colors: the 12 tile colors, as well as white,
  indicating the absence of a tile. We use the label 0 for the white
  regions.}
\label{fig:internal-space}
\end{figure}

To construct a hat tiling, overlay the triangular grid of
Figure~\ref{fig:triangular-grid} on the image of
Figure~\ref{fig:internal-space}. You may shift the grid left, right,
up, and down, but you may not rotate it. Place the grid in such a way
that none of the grid points fall on the boundaries of the colored
regions of Figure~\ref{fig:internal-space}. Then each grid point lies
in a unique colored region. If a grid point's color corresponds to one
of the 12 tile orientations, anchor a corresponding tile there. If a
grid point lands in a white region, no tile is anchored at it.
Figure~\ref{fig:example} shows an example of the triangular grid
overlaid on the colored image, and the corresponding tiling.

In the remainder of this paper, we explain in more detail how
Figure~\ref{fig:internal-space} is constructed, and why this
construction works.

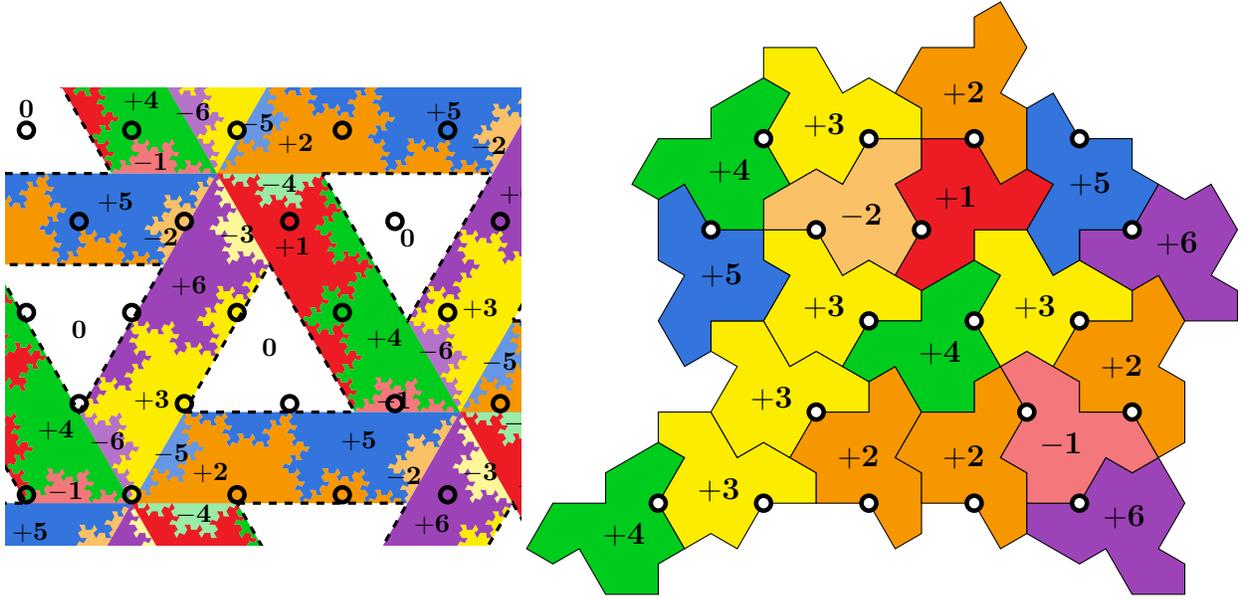
\begin{figure}
\begin{center}
\begin{tikzpicture}[
        scale=1.4
    ]
    \tikzstyle{every node}=[font=\large]

\begin{scope}
    \tikzstyle{every node}=[font=\normalsize]
    \clip (-1.2,-.4) rectangle (3.7, 3.95);
    \begin{pgfonlayer}{triangles}
      \clip (-1.2,-.4) rectangle (3.7, 3.95);
    \end{pgfonlayer}

    \foreach \a in {-1,0,1,2,3}
    \foreach \b in {-1,0,1,2}{
    \begin{scope}[xshift=\a*\Lxa cm + \b*\Lya cm,
                yshift=\a*\Lxb cm + \b*\Lyb cm]
            \FundamentalDomainAroundOrigin
            \FundamentalDomainLabellingFigureFive
    \end{scope}}

    \foreach \a in {-3,-2,-1,0,1,2,3}
    \foreach \b in {0,1,2,3,4}{
        \node[whitedotnone] at (\a+.5*\b,0.866025*\b+.08) {};
    }
\end{scope}

\begin{scope}[xshift=5cm]
\smith{0}{0}{+4}  
\smith{1}{0}{+3}  
\smith{2}{0}{+2}  
\smith{3}{0}{+2}  
\smith{4}{0}{+6}  
\smith{1.5}{0.866025}{+3}  
\smith{3.5}{0.866025}{-1}  
\smith{4.5}{0.866025}{+2}  
\smith{2.0}{0.866025*2}{+3}  
\smith{3.0}{0.866025*2}{+4}  
\smith{4.0}{0.866025*2}{+3}  
\smith{0.5}{0.866025*3}{+5}  
\smith{1.5}{0.866025*3}{-2}  
\smith{2.5}{0.866025*3}{+1}  
\smith{4.5}{0.866025*3}{+6}  
\smith{1.0}{0.866025*4}{+4}  
\smith{2.0}{0.866025*4}{+3}  
\smith{3.0}{0.866025*4}{+2}  
\smith{4.0}{0.866025*4}{+5}  
\end{scope}
\end{tikzpicture}
\end{center}

\caption{A valid pattern obtained by overlaying the triangular
  grid on the Markov partition.}
\label{fig:example}
\end{figure}
\section{Background and related work}

A tiling is a covering of the plane by isometric copies of a set of
polygons, called \emph{tiles}, that do not intersect (except on their boundaries).
Tilings are often periodic, that is, invariant under a nonzero translation.
A set of tiles is called \emph{aperiodic} if it tiles the plane but never
periodically. 
The first aperiodic sets were discovered by Berger in the 1960's
\cite{MR2939561,MR0216954} followed by Penrose in the 1970's
\cite{penrose_role_1974}.
It is worth noting that Penrose tilings resemble patterns observed in medieval
Islamic architecture constructed five centuries earlier \cite{Lu2007}.
The simplest version of Penrose tilings is made of only two tiles.  The existence of
an aperiodic monotile, i.e., of an aperiodic set made of a single tile, was an open question for 40 years until the recent
discovery of the hat tile \cite{MR4770585}. The history and theory of aperiodic
tilings is vast and will not be further described here.
For more information, we refer the reader to these books 
{\cite{MR857454,MR3136260,zbMATH00768067,zbMATH06785522}}. 
For additional work on the hat monotile and its family, see also
{\cite{zbMATH08141691,PhysRevB.108.224109,zbMATH08135691,zbMATH08105556,arXiv:2306.06512}}.

Our construction imitates a construction of valid
tilings for Jeandel and Rao's aperiodic set of 11 Wang tiles
\cite{MR4210631}. In \cite{MR4213162}, the first author showed
that valid Jeandel-Rao tilings can be
generated by coding the orbit of a $\Z^2$-action by a polygonal
partition of the torus. The method of the present
paper is the same, but the coding partition is partly fractal and
partly polygonal, which makes it unusual and interesting.

Anchors are also sometimes called \emph{control points}
\cite{MR3136260}. In contrast with \cite{zbMATH08105556}, where anchor
points lie inside or outside of the tiles, our placement of the anchor is
on a particular vertex on the boundary of the hat tile.

\section{Our construction, in more detail}

\subsection{Coordinates}

For our purposes, it is convenient to identify the plane with the set
of complex numbers. Let $\xi = \exp(\pi i/3)$ be the 6th root of
unity. Then every point in the plane can be written as $z =
a+b\xi$. The points of the triangular integer grid
(Figure~\ref{fig:triangular-grid}) are exactly the points with integer
coordinates, i.e., points of the form $z=a+b\xi$ where $a,b\in\Z$.

\subsection{The hat fractal}

The regions of Figure~\ref{fig:internal-space} are bounded by straight
lines and fractal curves. The fractal is constructed from two types
of segments by the recursive substitution rules shown in
Figure~\ref{fig:fractal-rules}(a). Some steps of the recursion are
shown in Figure~\ref{fig:fractal-rules}(b).  Note that the fractal
substitution rules are such that the fractal's complement is tiled
with triangles of various sizes (shown as gray triangles in
Figure~\ref{fig:fractal-rules}). With each application of the
substitution rules, the area covered by the blue and red segments
decreases by a constant factor. It follows that the fractal curve
obtained as the limit of this process has measure 0. Also, as an
intersection of infinitely many closed sets, the fractal curve is
closed. For each point that doesn't lie on this curve, it is easy to
compute which side of the curve is it on, because it will lie in one
of the (light or dark) gray triangles. 
Alternatively, the hat fractal can be defined from a substitution
on line segments as shown in Figure~\ref{fig:fractal-rules}(c).

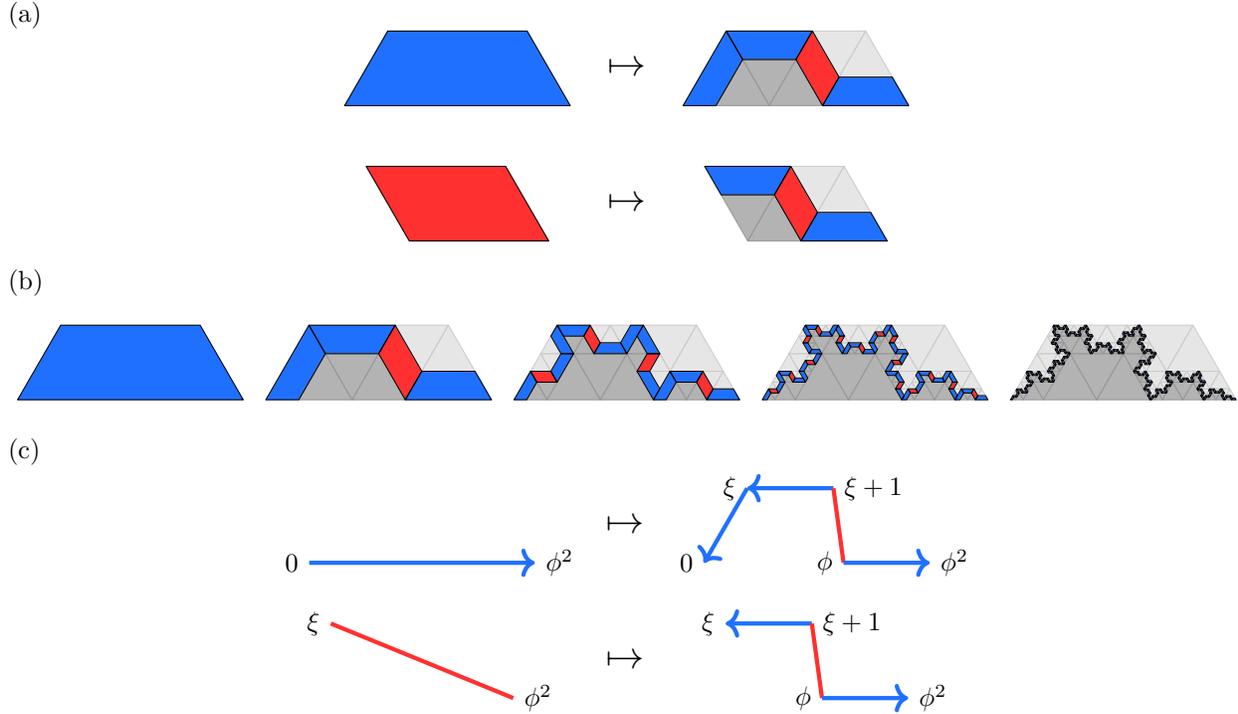
\begin{figure}
  (a)
  \[
  \fractal
  \begin{tikzpicture}[scale=3]
    \begin{scope}
      \Ar{0.000000}{0.000000}{0}{0}
    \end{scope}
    \path (1.25,0.165396) node[font=\Large] {$\mapsto$};
    \begin{scope}[xshift=1.5cm]
      \Td{0.381966}{0.000000}{0}{3}
      \Td{0.381966}{0.000000}{60}{3}
      \Td{0.381966}{0.000000}{120}{3}
      \Al{0.000000}{0.000000}{60}{2}
      \Al{0.190983}{0.330792}{0}{2}
      \C{0.572949}{0.330792}{0}{2}
      \Ar{0.618034}{0.000000}{0}{2}
      \Tl{0.809017}{0.330792}{240}{3}
      \Tl{0.809017}{0.330792}{300}{3}
    \end{scope}
    \begin{scope}[yshift=-0.6cm]
      \begin{scope}[xshift=-0.095492cm]
        \C{0.190983}{0.330792}{60}{0}
      \end{scope}
      \path (1.25,0.165396) node[font=\Large] {$\mapsto$};
      \begin{scope}[xshift=1.5cm,xshift=-0.095492cm]
        \Td{0.381966}{0.000000}{0}{3}
        \Td{0.381966}{0.000000}{60}{3}
        \Al{0.190983}{0.330792}{0}{2}
        \C{0.572949}{0.330792}{0}{2}
        \Ar{0.618034}{0.000000}{0}{2}
        \Tl{0.809017}{0.330792}{240}{3}
        \Tl{0.809017}{0.330792}{300}{3}
      \end{scope}
    \end{scope}
  \end{tikzpicture}
  \]
  (b)
  \[
  \fractal
  \begin{tikzpicture}[scale=3]
    \begin{scope}
      \Ar{0.000000}{0.000000}{0}{0}
    \end{scope}
    \begin{scope}[xshift=1.1cm]
      \Td{0.381966}{0.000000}{0}{3}
      \Td{0.381966}{0.000000}{60}{3}
      \Td{0.381966}{0.000000}{120}{3}
      \Al{0.000000}{0.000000}{60}{2}
      \Al{0.190983}{0.330792}{0}{2}
      \C{0.572949}{0.330792}{0}{2}
      \Ar{0.618034}{0.000000}{0}{2}
      \Tl{0.809017}{0.330792}{240}{3}
      \Tl{0.809017}{0.330792}{300}{3}
    \end{scope}
    \begin{scope}[xshift=2.2cm]
      \Td{0.381966}{0.000000}{0}{3}
      \Td{0.381966}{0.000000}{60}{3}
      \Td{0.381966}{0.000000}{120}{3}
      \Td{0.145898}{0.000000}{60}{5}
      \Td{0.145898}{0.000000}{120}{5}
      \Al{0.000000}{0.000000}{60}{4}
      \C{0.072949}{0.126351}{60}{4}
      \Ar{0.190983}{0.078089}{60}{4}
      \Tl{0.118034}{0.204441}{300}{5}
      \Tl{0.118034}{0.204441}{0}{5}
      \Tl{0.118034}{0.204441}{60}{5}
      \Ar{0.263932}{0.204441}{120}{4}
      \Td{0.263932}{0.204441}{0}{5}
      \Td{0.263932}{0.204441}{60}{5}
      \Al{0.190983}{0.330792}{0}{4}
      \C{0.336881}{0.330792}{0}{4}
      \Ar{0.354102}{0.204441}{0}{4}
      \Tl{0.427051}{0.330792}{240}{5}
      \Tl{0.427051}{0.330792}{300}{5}
      \Tl{0.427051}{0.330792}{0}{5}
      \Ar{0.500000}{0.204441}{60}{4}
      \Td{0.500000}{0.204441}{300}{5}
      \Td{0.500000}{0.204441}{0}{5}
      \Al{0.572949}{0.330792}{300}{4}
      \C{0.645898}{0.204441}{300}{4}
      \Ar{0.545085}{0.126351}{300}{4}
      \Tl{0.690983}{0.126351}{180}{5}
      \Tl{0.690983}{0.126351}{240}{5}
      \Td{0.763932}{-0.000000}{0}{5}
      \Td{0.763932}{-0.000000}{60}{5}
      \Td{0.763932}{-0.000000}{120}{5}
      \Al{0.618034}{-0.000000}{60}{4}
      \Al{0.690983}{0.126351}{0}{4}
      \C{0.836881}{0.126351}{0}{4}
      \Ar{0.854102}{-0.000000}{0}{4}
      \Tl{0.927051}{0.126351}{240}{5}
      \Tl{0.927051}{0.126351}{300}{5}
      \Tl{0.809017}{0.330792}{240}{3}
      \Tl{0.809017}{0.330792}{300}{3}
    \end{scope}
    \begin{scope}[xshift=3.3cm]
      \tikzset{mypath/.style={draw=black,line width=0.2}}
      \Td{0.381966}{0.000000}{0}{3}
      \Td{0.381966}{0.000000}{60}{3}
      \Td{0.381966}{0.000000}{120}{3}
      \Td{0.145898}{0.000000}{60}{5}
      \Td{0.145898}{0.000000}{120}{5}
      \Td{0.055728}{0.000000}{60}{7}
      \Td{0.055728}{0.000000}{120}{7}
      \Al{0.000000}{0.000000}{60}{6}
      \C{0.027864}{0.048262}{60}{6}
      \Ar{0.072949}{0.029828}{60}{6}
      \Tl{0.045085}{0.078089}{300}{7}
      \Tl{0.045085}{0.078089}{0}{7}
      \Tl{0.045085}{0.078089}{60}{7}
      \Ar{0.100813}{0.078089}{120}{6}
      \Td{0.100813}{0.078089}{0}{7}
      \Td{0.100813}{0.078089}{60}{7}
      \Al{0.072949}{0.126351}{0}{6}
      \C{0.128677}{0.126351}{0}{6}
      \Ar{0.135255}{0.078089}{0}{6}
      \Tl{0.163119}{0.126351}{240}{7}
      \Tl{0.163119}{0.126351}{300}{7}
      \Td{0.218847}{0.126351}{60}{7}
      \Td{0.218847}{0.126351}{120}{7}
      \Td{0.218847}{0.126351}{180}{7}
      \Al{0.190983}{0.078089}{120}{6}
      \Al{0.163119}{0.126351}{60}{6}
      \C{0.190983}{0.174613}{60}{6}
      \Ar{0.236068}{0.156179}{60}{6}
      \Tl{0.208204}{0.204441}{300}{7}
      \Tl{0.208204}{0.204441}{0}{7}
      \Tl{0.118034}{0.204441}{300}{5}
      \Tl{0.118034}{0.204441}{0}{5}
      \Tl{0.118034}{0.204441}{60}{5}
      \Td{0.236068}{0.252703}{120}{7}
      \Td{0.236068}{0.252703}{180}{7}
      \Td{0.236068}{0.252703}{240}{7}
      \Al{0.263932}{0.204441}{180}{6}
      \Al{0.208204}{0.204441}{120}{6}
      \C{0.180340}{0.252703}{120}{6}
      \Ar{0.218847}{0.282530}{120}{6}
      \Tl{0.163119}{0.282530}{0}{7}
      \Tl{0.163119}{0.282530}{60}{7}
      \Td{0.263932}{0.204441}{0}{5}
      \Td{0.263932}{0.204441}{60}{5}
      \Td{0.218847}{0.282530}{0}{7}
      \Td{0.218847}{0.282530}{60}{7}
      \Al{0.190983}{0.330792}{0}{6}
      \C{0.246711}{0.330792}{0}{6}
      \Ar{0.253289}{0.282530}{0}{6}
      \Tl{0.281153}{0.330792}{240}{7}
      \Tl{0.281153}{0.330792}{300}{7}
      \Tl{0.281153}{0.330792}{0}{7}
      \Ar{0.309017}{0.282530}{60}{6}
      \Td{0.309017}{0.282530}{300}{7}
      \Td{0.309017}{0.282530}{0}{7}
      \Al{0.336881}{0.330792}{300}{6}
      \C{0.364745}{0.282530}{300}{6}
      \Ar{0.326238}{0.252703}{300}{6}
      \Tl{0.381966}{0.252703}{180}{7}
      \Tl{0.381966}{0.252703}{240}{7}
      \Td{0.409830}{0.204441}{0}{7}
      \Td{0.409830}{0.204441}{60}{7}
      \Td{0.409830}{0.204441}{120}{7}
      \Al{0.354102}{0.204441}{60}{6}
      \Al{0.381966}{0.252703}{0}{6}
      \C{0.437694}{0.252703}{0}{6}
      \Ar{0.444272}{0.204441}{0}{6}
      \Tl{0.472136}{0.252703}{240}{7}
      \Tl{0.472136}{0.252703}{300}{7}
      \Tl{0.427051}{0.330792}{240}{5}
      \Tl{0.427051}{0.330792}{300}{5}
      \Tl{0.427051}{0.330792}{0}{5}
      \Td{0.527864}{0.252703}{60}{7}
      \Td{0.527864}{0.252703}{120}{7}
      \Td{0.527864}{0.252703}{180}{7}
      \Al{0.500000}{0.204441}{120}{6}
      \Al{0.472136}{0.252703}{60}{6}
      \C{0.500000}{0.300965}{60}{6}
      \Ar{0.545085}{0.282530}{60}{6}
      \Tl{0.517221}{0.330792}{300}{7}
      \Tl{0.517221}{0.330792}{0}{7}
      \Td{0.500000}{0.204441}{300}{5}
      \Td{0.500000}{0.204441}{0}{5}
      \Td{0.545085}{0.282530}{300}{7}
      \Td{0.545085}{0.282530}{0}{7}
      \Al{0.572949}{0.330792}{300}{6}
      \C{0.600813}{0.282530}{300}{6}
      \Ar{0.562306}{0.252703}{300}{6}
      \Tl{0.618034}{0.252703}{180}{7}
      \Tl{0.618034}{0.252703}{240}{7}
      \Tl{0.618034}{0.252703}{300}{7}
      \Ar{0.590170}{0.204441}{0}{6}
      \Td{0.590170}{0.204441}{240}{7}
      \Td{0.590170}{0.204441}{300}{7}
      \Al{0.645898}{0.204441}{240}{6}
      \C{0.618034}{0.156179}{240}{6}
      \Ar{0.572949}{0.174613}{240}{6}
      \Tl{0.600813}{0.126351}{120}{7}
      \Tl{0.600813}{0.126351}{180}{7}
      \Td{0.572949}{0.078089}{300}{7}
      \Td{0.572949}{0.078089}{0}{7}
      \Td{0.572949}{0.078089}{60}{7}
      \Al{0.545085}{0.126351}{0}{6}
      \Al{0.600813}{0.126351}{300}{6}
      \C{0.628677}{0.078089}{300}{6}
      \Ar{0.590170}{0.048262}{300}{6}
      \Tl{0.645898}{0.048262}{180}{7}
      \Tl{0.645898}{0.048262}{240}{7}
      \Tl{0.690983}{0.126351}{180}{5}
      \Tl{0.690983}{0.126351}{240}{5}
      \Td{0.763932}{-0.000000}{0}{5}
      \Td{0.763932}{-0.000000}{60}{5}
      \Td{0.763932}{-0.000000}{120}{5}
      \Td{0.673762}{-0.000000}{60}{7}
      \Td{0.673762}{-0.000000}{120}{7}
      \Al{0.618034}{-0.000000}{60}{6}
      \C{0.645898}{0.048262}{60}{6}
      \Ar{0.690983}{0.029828}{60}{6}
      \Tl{0.663119}{0.078089}{300}{7}
      \Tl{0.663119}{0.078089}{0}{7}
      \Tl{0.663119}{0.078089}{60}{7}
      \Ar{0.718847}{0.078089}{120}{6}
      \Td{0.718847}{0.078089}{0}{7}
      \Td{0.718847}{0.078089}{60}{7}
      \Al{0.690983}{0.126351}{0}{6}
      \C{0.746711}{0.126351}{0}{6}
      \Ar{0.753289}{0.078089}{0}{6}
      \Tl{0.781153}{0.126351}{240}{7}
      \Tl{0.781153}{0.126351}{300}{7}
      \Tl{0.781153}{0.126351}{0}{7}
      \Ar{0.809017}{0.078089}{60}{6}
      \Td{0.809017}{0.078089}{300}{7}
      \Td{0.809017}{0.078089}{0}{7}
      \Al{0.836881}{0.126351}{300}{6}
      \C{0.864745}{0.078089}{300}{6}
      \Ar{0.826238}{0.048262}{300}{6}
      \Tl{0.881966}{0.048262}{180}{7}
      \Tl{0.881966}{0.048262}{240}{7}
      \Td{0.909830}{-0.000000}{0}{7}
      \Td{0.909830}{-0.000000}{60}{7}
      \Td{0.909830}{-0.000000}{120}{7}
      \Al{0.854102}{-0.000000}{60}{6}
      \Al{0.881966}{0.048262}{0}{6}
      \C{0.937694}{0.048262}{0}{6}
      \Ar{0.944272}{-0.000000}{0}{6}
      \Tl{0.972136}{0.048262}{240}{7}
      \Tl{0.972136}{0.048262}{300}{7}
      \Tl{0.927051}{0.126351}{240}{5}
      \Tl{0.927051}{0.126351}{300}{5}
      \Tl{0.809017}{0.330792}{240}{3}
      \Tl{0.809017}{0.330792}{300}{3}
    \end{scope}
    \begin{scope}[xshift=4.4cm]
      \tikzset{mypath/.style={draw=black,line width=0.1}}
      \Td{0.381966}{0.000000}{0}{3}
      \Td{0.381966}{0.000000}{60}{3}
      \Td{0.381966}{0.000000}{120}{3}
      \Td{0.145898}{0.000000}{60}{5}
      \Td{0.145898}{0.000000}{120}{5}
      \Td{0.055728}{0.000000}{60}{7}
      \Td{0.055728}{0.000000}{120}{7}
      \Td{0.021286}{0.000000}{60}{9}
      \Td{0.021286}{0.000000}{120}{9}
      \Al{0.000000}{0.000000}{60}{8}
      \C{0.010643}{0.018434}{60}{8}
      \Ar{0.027864}{0.011393}{60}{8}
      \Tl{0.017221}{0.029828}{300}{9}
      \Tl{0.017221}{0.029828}{0}{9}
      \Tl{0.017221}{0.029828}{60}{9}
      \Ar{0.038507}{0.029828}{120}{8}
      \Td{0.038507}{0.029828}{0}{9}
      \Td{0.038507}{0.029828}{60}{9}
      \Al{0.027864}{0.048262}{0}{8}
      \C{0.049150}{0.048262}{0}{8}
      \Ar{0.051663}{0.029828}{0}{8}
      \Tl{0.062306}{0.048262}{240}{9}
      \Tl{0.062306}{0.048262}{300}{9}
      \Td{0.083592}{0.048262}{60}{9}
      \Td{0.083592}{0.048262}{120}{9}
      \Td{0.083592}{0.048262}{180}{9}
      \Al{0.072949}{0.029828}{120}{8}
      \Al{0.062306}{0.048262}{60}{8}
      \C{0.072949}{0.066696}{60}{8}
      \Ar{0.090170}{0.059655}{60}{8}
      \Tl{0.079527}{0.078089}{300}{9}
      \Tl{0.079527}{0.078089}{0}{9}
      \Tl{0.045085}{0.078089}{300}{7}
      \Tl{0.045085}{0.078089}{0}{7}
      \Tl{0.045085}{0.078089}{60}{7}
      \Td{0.090170}{0.096524}{120}{9}
      \Td{0.090170}{0.096524}{180}{9}
      \Td{0.090170}{0.096524}{240}{9}
      \Al{0.100813}{0.078089}{180}{8}
      \Al{0.079527}{0.078089}{120}{8}
      \C{0.068884}{0.096524}{120}{8}
      \Ar{0.083592}{0.107917}{120}{8}
      \Tl{0.062306}{0.107917}{0}{9}
      \Tl{0.062306}{0.107917}{60}{9}
      \Td{0.100813}{0.078089}{0}{7}
      \Td{0.100813}{0.078089}{60}{7}
      \Td{0.083592}{0.107917}{0}{9}
      \Td{0.083592}{0.107917}{60}{9}
      \Al{0.072949}{0.126351}{0}{8}
      \C{0.094235}{0.126351}{0}{8}
      \Ar{0.096748}{0.107917}{0}{8}
      \Tl{0.107391}{0.126351}{240}{9}
      \Tl{0.107391}{0.126351}{300}{9}
      \Tl{0.107391}{0.126351}{0}{9}
      \Ar{0.118034}{0.107917}{60}{8}
      \Td{0.118034}{0.107917}{300}{9}
      \Td{0.118034}{0.107917}{0}{9}
      \Al{0.128677}{0.126351}{300}{8}
      \C{0.139320}{0.107917}{300}{8}
      \Ar{0.124612}{0.096524}{300}{8}
      \Tl{0.145898}{0.096524}{180}{9}
      \Tl{0.145898}{0.096524}{240}{9}
      \Td{0.156541}{0.078089}{0}{9}
      \Td{0.156541}{0.078089}{60}{9}
      \Td{0.156541}{0.078089}{120}{9}
      \Al{0.135255}{0.078089}{60}{8}
      \Al{0.145898}{0.096524}{0}{8}
      \C{0.167184}{0.096524}{0}{8}
      \Ar{0.169697}{0.078089}{0}{8}
      \Tl{0.180340}{0.096524}{240}{9}
      \Tl{0.180340}{0.096524}{300}{9}
      \Tl{0.163119}{0.126351}{240}{7}
      \Tl{0.163119}{0.126351}{300}{7}
      \Td{0.218847}{0.126351}{60}{7}
      \Td{0.218847}{0.126351}{120}{7}
      \Td{0.218847}{0.126351}{180}{7}
      \Td{0.201626}{0.096524}{120}{9}
      \Td{0.201626}{0.096524}{180}{9}
      \Al{0.190983}{0.078089}{120}{8}
      \C{0.180340}{0.096524}{120}{8}
      \Ar{0.195048}{0.107917}{120}{8}
      \Tl{0.173762}{0.107917}{0}{9}
      \Tl{0.173762}{0.107917}{60}{9}
      \Tl{0.173762}{0.107917}{120}{9}
      \Ar{0.184405}{0.126351}{180}{8}
      \Td{0.184405}{0.126351}{60}{9}
      \Td{0.184405}{0.126351}{120}{9}
      \Al{0.163119}{0.126351}{60}{8}
      \C{0.173762}{0.144786}{60}{8}
      \Ar{0.190983}{0.137745}{60}{8}
      \Tl{0.180340}{0.156179}{300}{9}
      \Tl{0.180340}{0.156179}{0}{9}
      \Tl{0.180340}{0.156179}{60}{9}
      \Ar{0.201626}{0.156179}{120}{8}
      \Td{0.201626}{0.156179}{0}{9}
      \Td{0.201626}{0.156179}{60}{9}
      \Al{0.190983}{0.174613}{0}{8}
      \C{0.212269}{0.174613}{0}{8}
      \Ar{0.214782}{0.156179}{0}{8}
      \Tl{0.225425}{0.174613}{240}{9}
      \Tl{0.225425}{0.174613}{300}{9}
      \Td{0.246711}{0.174613}{60}{9}
      \Td{0.246711}{0.174613}{120}{9}
      \Td{0.246711}{0.174613}{180}{9}
      \Al{0.236068}{0.156179}{120}{8}
      \Al{0.225425}{0.174613}{60}{8}
      \C{0.236068}{0.193048}{60}{8}
      \Ar{0.253289}{0.186006}{60}{8}
      \Tl{0.242646}{0.204441}{300}{9}
      \Tl{0.242646}{0.204441}{0}{9}
      \Tl{0.208204}{0.204441}{300}{7}
      \Tl{0.208204}{0.204441}{0}{7}
      \Tl{0.118034}{0.204441}{300}{5}
      \Tl{0.118034}{0.204441}{0}{5}
      \Tl{0.118034}{0.204441}{60}{5}
      \Td{0.236068}{0.252703}{120}{7}
      \Td{0.236068}{0.252703}{180}{7}
      \Td{0.236068}{0.252703}{240}{7}
      \Td{0.253289}{0.222875}{180}{9}
      \Td{0.253289}{0.222875}{240}{9}
      \Al{0.263932}{0.204441}{180}{8}
      \C{0.242646}{0.204441}{180}{8}
      \Ar{0.240133}{0.222875}{180}{8}
      \Tl{0.229490}{0.204441}{60}{9}
      \Tl{0.229490}{0.204441}{120}{9}
      \Tl{0.229490}{0.204441}{180}{9}
      \Ar{0.218847}{0.222875}{240}{8}
      \Td{0.218847}{0.222875}{120}{9}
      \Td{0.218847}{0.222875}{180}{9}
      \Al{0.208204}{0.204441}{120}{8}
      \C{0.197561}{0.222875}{120}{8}
      \Ar{0.212269}{0.234268}{120}{8}
      \Tl{0.190983}{0.234268}{0}{9}
      \Tl{0.190983}{0.234268}{60}{9}
      \Tl{0.190983}{0.234268}{120}{9}
      \Ar{0.201626}{0.252703}{180}{8}
      \Td{0.201626}{0.252703}{60}{9}
      \Td{0.201626}{0.252703}{120}{9}
      \Al{0.180340}{0.252703}{60}{8}
      \C{0.190983}{0.271137}{60}{8}
      \Ar{0.208204}{0.264096}{60}{8}
      \Tl{0.197561}{0.282530}{300}{9}
      \Tl{0.197561}{0.282530}{0}{9}
      \Td{0.208204}{0.300965}{120}{9}
      \Td{0.208204}{0.300965}{180}{9}
      \Td{0.208204}{0.300965}{240}{9}
      \Al{0.218847}{0.282530}{180}{8}
      \Al{0.197561}{0.282530}{120}{8}
      \C{0.186918}{0.300965}{120}{8}
      \Ar{0.201626}{0.312358}{120}{8}
      \Tl{0.180340}{0.312358}{0}{9}
      \Tl{0.180340}{0.312358}{60}{9}
      \Tl{0.163119}{0.282530}{0}{7}
      \Tl{0.163119}{0.282530}{60}{7}
      \Td{0.263932}{0.204441}{0}{5}
      \Td{0.263932}{0.204441}{60}{5}
      \Td{0.218847}{0.282530}{0}{7}
      \Td{0.218847}{0.282530}{60}{7}
      \Td{0.201626}{0.312358}{0}{9}
      \Td{0.201626}{0.312358}{60}{9}
      \Al{0.190983}{0.330792}{0}{8}
      \C{0.212269}{0.330792}{0}{8}
      \Ar{0.214782}{0.312358}{0}{8}
      \Tl{0.225425}{0.330792}{240}{9}
      \Tl{0.225425}{0.330792}{300}{9}
      \Tl{0.225425}{0.330792}{0}{9}
      \Ar{0.236068}{0.312358}{60}{8}
      \Td{0.236068}{0.312358}{300}{9}
      \Td{0.236068}{0.312358}{0}{9}
      \Al{0.246711}{0.330792}{300}{8}
      \C{0.257354}{0.312358}{300}{8}
      \Ar{0.242646}{0.300965}{300}{8}
      \Tl{0.263932}{0.300965}{180}{9}
      \Tl{0.263932}{0.300965}{240}{9}
      \Td{0.274575}{0.282530}{0}{9}
      \Td{0.274575}{0.282530}{60}{9}
      \Td{0.274575}{0.282530}{120}{9}
      \Al{0.253289}{0.282530}{60}{8}
      \Al{0.263932}{0.300965}{0}{8}
      \C{0.285218}{0.300965}{0}{8}
      \Ar{0.287731}{0.282530}{0}{8}
      \Tl{0.298374}{0.300965}{240}{9}
      \Tl{0.298374}{0.300965}{300}{9}
      \Tl{0.281153}{0.330792}{240}{7}
      \Tl{0.281153}{0.330792}{300}{7}
      \Tl{0.281153}{0.330792}{0}{7}
      \Td{0.319660}{0.300965}{60}{9}
      \Td{0.319660}{0.300965}{120}{9}
      \Td{0.319660}{0.300965}{180}{9}
      \Al{0.309017}{0.282530}{120}{8}
      \Al{0.298374}{0.300965}{60}{8}
      \C{0.309017}{0.319399}{60}{8}
      \Ar{0.326238}{0.312358}{60}{8}
      \Tl{0.315595}{0.330792}{300}{9}
      \Tl{0.315595}{0.330792}{0}{9}
      \Td{0.309017}{0.282530}{300}{7}
      \Td{0.309017}{0.282530}{0}{7}
      \Td{0.326238}{0.312358}{300}{9}
      \Td{0.326238}{0.312358}{0}{9}
      \Al{0.336881}{0.330792}{300}{8}
      \C{0.347524}{0.312358}{300}{8}
      \Ar{0.332816}{0.300965}{300}{8}
      \Tl{0.354102}{0.300965}{180}{9}
      \Tl{0.354102}{0.300965}{240}{9}
      \Tl{0.354102}{0.300965}{300}{9}
      \Ar{0.343459}{0.282530}{0}{8}
      \Td{0.343459}{0.282530}{240}{9}
      \Td{0.343459}{0.282530}{300}{9}
      \Al{0.364745}{0.282530}{240}{8}
      \C{0.354102}{0.264096}{240}{8}
      \Ar{0.336881}{0.271137}{240}{8}
      \Tl{0.347524}{0.252703}{120}{9}
      \Tl{0.347524}{0.252703}{180}{9}
      \Td{0.336881}{0.234268}{300}{9}
      \Td{0.336881}{0.234268}{0}{9}
      \Td{0.336881}{0.234268}{60}{9}
      \Al{0.326238}{0.252703}{0}{8}
      \Al{0.347524}{0.252703}{300}{8}
      \C{0.358167}{0.234268}{300}{8}
      \Ar{0.343459}{0.222875}{300}{8}
      \Tl{0.364745}{0.222875}{180}{9}
      \Tl{0.364745}{0.222875}{240}{9}
      \Tl{0.381966}{0.252703}{180}{7}
      \Tl{0.381966}{0.252703}{240}{7}
      \Td{0.409830}{0.204441}{0}{7}
      \Td{0.409830}{0.204441}{60}{7}
      \Td{0.409830}{0.204441}{120}{7}
      \Td{0.375388}{0.204441}{60}{9}
      \Td{0.375388}{0.204441}{120}{9}
      \Al{0.354102}{0.204441}{60}{8}
      \C{0.364745}{0.222875}{60}{8}
      \Ar{0.381966}{0.215834}{60}{8}
      \Tl{0.371323}{0.234268}{300}{9}
      \Tl{0.371323}{0.234268}{0}{9}
      \Tl{0.371323}{0.234268}{60}{9}
      \Ar{0.392609}{0.234268}{120}{8}
      \Td{0.392609}{0.234268}{0}{9}
      \Td{0.392609}{0.234268}{60}{9}
      \Al{0.381966}{0.252703}{0}{8}
      \C{0.403252}{0.252703}{0}{8}
      \Ar{0.405765}{0.234268}{0}{8}
      \Tl{0.416408}{0.252703}{240}{9}
      \Tl{0.416408}{0.252703}{300}{9}
      \Tl{0.416408}{0.252703}{0}{9}
      \Ar{0.427051}{0.234268}{60}{8}
      \Td{0.427051}{0.234268}{300}{9}
      \Td{0.427051}{0.234268}{0}{9}
      \Al{0.437694}{0.252703}{300}{8}
      \C{0.448337}{0.234268}{300}{8}
      \Ar{0.433629}{0.222875}{300}{8}
      \Tl{0.454915}{0.222875}{180}{9}
      \Tl{0.454915}{0.222875}{240}{9}
      \Td{0.465558}{0.204441}{0}{9}
      \Td{0.465558}{0.204441}{60}{9}
      \Td{0.465558}{0.204441}{120}{9}
      \Al{0.444272}{0.204441}{60}{8}
      \Al{0.454915}{0.222875}{0}{8}
      \C{0.476201}{0.222875}{0}{8}
      \Ar{0.478714}{0.204441}{0}{8}
      \Tl{0.489357}{0.222875}{240}{9}
      \Tl{0.489357}{0.222875}{300}{9}
      \Tl{0.472136}{0.252703}{240}{7}
      \Tl{0.472136}{0.252703}{300}{7}
      \Tl{0.427051}{0.330792}{240}{5}
      \Tl{0.427051}{0.330792}{300}{5}
      \Tl{0.427051}{0.330792}{0}{5}
      \Td{0.527864}{0.252703}{60}{7}
      \Td{0.527864}{0.252703}{120}{7}
      \Td{0.527864}{0.252703}{180}{7}
      \Td{0.510643}{0.222875}{120}{9}
      \Td{0.510643}{0.222875}{180}{9}
      \Al{0.500000}{0.204441}{120}{8}
      \C{0.489357}{0.222875}{120}{8}
      \Ar{0.504065}{0.234268}{120}{8}
      \Tl{0.482779}{0.234268}{0}{9}
      \Tl{0.482779}{0.234268}{60}{9}
      \Tl{0.482779}{0.234268}{120}{9}
      \Ar{0.493422}{0.252703}{180}{8}
      \Td{0.493422}{0.252703}{60}{9}
      \Td{0.493422}{0.252703}{120}{9}
      \Al{0.472136}{0.252703}{60}{8}
      \C{0.482779}{0.271137}{60}{8}
      \Ar{0.500000}{0.264096}{60}{8}
      \Tl{0.489357}{0.282530}{300}{9}
      \Tl{0.489357}{0.282530}{0}{9}
      \Tl{0.489357}{0.282530}{60}{9}
      \Ar{0.510643}{0.282530}{120}{8}
      \Td{0.510643}{0.282530}{0}{9}
      \Td{0.510643}{0.282530}{60}{9}
      \Al{0.500000}{0.300965}{0}{8}
      \C{0.521286}{0.300965}{0}{8}
      \Ar{0.523799}{0.282530}{0}{8}
      \Tl{0.534442}{0.300965}{240}{9}
      \Tl{0.534442}{0.300965}{300}{9}
      \Td{0.555728}{0.300965}{60}{9}
      \Td{0.555728}{0.300965}{120}{9}
      \Td{0.555728}{0.300965}{180}{9}
      \Al{0.545085}{0.282530}{120}{8}
      \Al{0.534442}{0.300965}{60}{8}
      \C{0.545085}{0.319399}{60}{8}
      \Ar{0.562306}{0.312358}{60}{8}
      \Tl{0.551663}{0.330792}{300}{9}
      \Tl{0.551663}{0.330792}{0}{9}
      \Tl{0.517221}{0.330792}{300}{7}
      \Tl{0.517221}{0.330792}{0}{7}
      \Td{0.500000}{0.204441}{300}{5}
      \Td{0.500000}{0.204441}{0}{5}
      \Td{0.545085}{0.282530}{300}{7}
      \Td{0.545085}{0.282530}{0}{7}
      \Td{0.562306}{0.312358}{300}{9}
      \Td{0.562306}{0.312358}{0}{9}
      \Al{0.572949}{0.330792}{300}{8}
      \C{0.583592}{0.312358}{300}{8}
      \Ar{0.568884}{0.300965}{300}{8}
      \Tl{0.590170}{0.300965}{180}{9}
      \Tl{0.590170}{0.300965}{240}{9}
      \Tl{0.590170}{0.300965}{300}{9}
      \Ar{0.579527}{0.282530}{0}{8}
      \Td{0.579527}{0.282530}{240}{9}
      \Td{0.579527}{0.282530}{300}{9}
      \Al{0.600813}{0.282530}{240}{8}
      \C{0.590170}{0.264096}{240}{8}
      \Ar{0.572949}{0.271137}{240}{8}
      \Tl{0.583592}{0.252703}{120}{9}
      \Tl{0.583592}{0.252703}{180}{9}
      \Td{0.572949}{0.234268}{300}{9}
      \Td{0.572949}{0.234268}{0}{9}
      \Td{0.572949}{0.234268}{60}{9}
      \Al{0.562306}{0.252703}{0}{8}
      \Al{0.583592}{0.252703}{300}{8}
      \C{0.594235}{0.234268}{300}{8}
      \Ar{0.579527}{0.222875}{300}{8}
      \Tl{0.600813}{0.222875}{180}{9}
      \Tl{0.600813}{0.222875}{240}{9}
      \Tl{0.618034}{0.252703}{180}{7}
      \Tl{0.618034}{0.252703}{240}{7}
      \Tl{0.618034}{0.252703}{300}{7}
      \Td{0.611456}{0.204441}{0}{9}
      \Td{0.611456}{0.204441}{60}{9}
      \Td{0.611456}{0.204441}{120}{9}
      \Al{0.590170}{0.204441}{60}{8}
      \Al{0.600813}{0.222875}{0}{8}
      \C{0.622099}{0.222875}{0}{8}
      \Ar{0.624612}{0.204441}{0}{8}
      \Tl{0.635255}{0.222875}{240}{9}
      \Tl{0.635255}{0.222875}{300}{9}
      \Td{0.590170}{0.204441}{240}{7}
      \Td{0.590170}{0.204441}{300}{7}
      \Td{0.624612}{0.204441}{240}{9}
      \Td{0.624612}{0.204441}{300}{9}
      \Al{0.645898}{0.204441}{240}{8}
      \C{0.635255}{0.186006}{240}{8}
      \Ar{0.618034}{0.193048}{240}{8}
      \Tl{0.628677}{0.174613}{120}{9}
      \Tl{0.628677}{0.174613}{180}{9}
      \Tl{0.628677}{0.174613}{240}{9}
      \Ar{0.607391}{0.174613}{300}{8}
      \Td{0.607391}{0.174613}{180}{9}
      \Td{0.607391}{0.174613}{240}{9}
      \Al{0.618034}{0.156179}{180}{8}
      \C{0.596748}{0.156179}{180}{8}
      \Ar{0.594235}{0.174613}{180}{8}
      \Tl{0.583592}{0.156179}{60}{9}
      \Tl{0.583592}{0.156179}{120}{9}
      \Td{0.562306}{0.156179}{240}{9}
      \Td{0.562306}{0.156179}{300}{9}
      \Td{0.562306}{0.156179}{0}{9}
      \Al{0.572949}{0.174613}{300}{8}
      \Al{0.583592}{0.156179}{240}{8}
      \C{0.572949}{0.137745}{240}{8}
      \Ar{0.555728}{0.144786}{240}{8}
      \Tl{0.566371}{0.126351}{120}{9}
      \Tl{0.566371}{0.126351}{180}{9}
      \Tl{0.600813}{0.126351}{120}{7}
      \Tl{0.600813}{0.126351}{180}{7}
      \Td{0.572949}{0.078089}{300}{7}
      \Td{0.572949}{0.078089}{0}{7}
      \Td{0.572949}{0.078089}{60}{7}
      \Td{0.555728}{0.107917}{0}{9}
      \Td{0.555728}{0.107917}{60}{9}
      \Al{0.545085}{0.126351}{0}{8}
      \C{0.566371}{0.126351}{0}{8}
      \Ar{0.568884}{0.107917}{0}{8}
      \Tl{0.579527}{0.126351}{240}{9}
      \Tl{0.579527}{0.126351}{300}{9}
      \Tl{0.579527}{0.126351}{0}{9}
      \Ar{0.590170}{0.107917}{60}{8}
      \Td{0.590170}{0.107917}{300}{9}
      \Td{0.590170}{0.107917}{0}{9}
      \Al{0.600813}{0.126351}{300}{8}
      \C{0.611456}{0.107917}{300}{8}
      \Ar{0.596748}{0.096524}{300}{8}
      \Tl{0.618034}{0.096524}{180}{9}
      \Tl{0.618034}{0.096524}{240}{9}
      \Tl{0.618034}{0.096524}{300}{9}
      \Ar{0.607391}{0.078089}{0}{8}
      \Td{0.607391}{0.078089}{240}{9}
      \Td{0.607391}{0.078089}{300}{9}
      \Al{0.628677}{0.078089}{240}{8}
      \C{0.618034}{0.059655}{240}{8}
      \Ar{0.600813}{0.066696}{240}{8}
      \Tl{0.611456}{0.048262}{120}{9}
      \Tl{0.611456}{0.048262}{180}{9}
      \Td{0.600813}{0.029828}{300}{9}
      \Td{0.600813}{0.029828}{0}{9}
      \Td{0.600813}{0.029828}{60}{9}
      \Al{0.590170}{0.048262}{0}{8}
      \Al{0.611456}{0.048262}{300}{8}
      \C{0.622099}{0.029828}{300}{8}
      \Ar{0.607391}{0.018434}{300}{8}
      \Tl{0.628677}{0.018434}{180}{9}
      \Tl{0.628677}{0.018434}{240}{9}
      \Tl{0.645898}{0.048262}{180}{7}
      \Tl{0.645898}{0.048262}{240}{7}
      \Tl{0.690983}{0.126351}{180}{5}
      \Tl{0.690983}{0.126351}{240}{5}
      \Td{0.763932}{-0.000000}{0}{5}
      \Td{0.763932}{-0.000000}{60}{5}
      \Td{0.763932}{-0.000000}{120}{5}
      \Td{0.673762}{-0.000000}{60}{7}
      \Td{0.673762}{-0.000000}{120}{7}
      \Td{0.639320}{-0.000000}{60}{9}
      \Td{0.639320}{-0.000000}{120}{9}
      \Al{0.618034}{-0.000000}{60}{8}
      \C{0.628677}{0.018434}{60}{8}
      \Ar{0.645898}{0.011393}{60}{8}
      \Tl{0.635255}{0.029828}{300}{9}
      \Tl{0.635255}{0.029828}{0}{9}
      \Tl{0.635255}{0.029828}{60}{9}
      \Ar{0.656541}{0.029828}{120}{8}
      \Td{0.656541}{0.029828}{0}{9}
      \Td{0.656541}{0.029828}{60}{9}
      \Al{0.645898}{0.048262}{0}{8}
      \C{0.667184}{0.048262}{0}{8}
      \Ar{0.669697}{0.029828}{0}{8}
      \Tl{0.680340}{0.048262}{240}{9}
      \Tl{0.680340}{0.048262}{300}{9}
      \Td{0.701626}{0.048262}{60}{9}
      \Td{0.701626}{0.048262}{120}{9}
      \Td{0.701626}{0.048262}{180}{9}
      \Al{0.690983}{0.029828}{120}{8}
      \Al{0.680340}{0.048262}{60}{8}
      \C{0.690983}{0.066696}{60}{8}
      \Ar{0.708204}{0.059655}{60}{8}
      \Tl{0.697561}{0.078089}{300}{9}
      \Tl{0.697561}{0.078089}{0}{9}
      \Tl{0.663119}{0.078089}{300}{7}
      \Tl{0.663119}{0.078089}{0}{7}
      \Tl{0.663119}{0.078089}{60}{7}
      \Td{0.708204}{0.096524}{120}{9}
      \Td{0.708204}{0.096524}{180}{9}
      \Td{0.708204}{0.096524}{240}{9}
      \Al{0.718847}{0.078089}{180}{8}
      \Al{0.697561}{0.078089}{120}{8}
      \C{0.686918}{0.096524}{120}{8}
      \Ar{0.701626}{0.107917}{120}{8}
      \Tl{0.680340}{0.107917}{0}{9}
      \Tl{0.680340}{0.107917}{60}{9}
      \Td{0.718847}{0.078089}{0}{7}
      \Td{0.718847}{0.078089}{60}{7}
      \Td{0.701626}{0.107917}{0}{9}
      \Td{0.701626}{0.107917}{60}{9}
      \Al{0.690983}{0.126351}{0}{8}
      \C{0.712269}{0.126351}{0}{8}
      \Ar{0.714782}{0.107917}{0}{8}
      \Tl{0.725425}{0.126351}{240}{9}
      \Tl{0.725425}{0.126351}{300}{9}
      \Tl{0.725425}{0.126351}{0}{9}
      \Ar{0.736068}{0.107917}{60}{8}
      \Td{0.736068}{0.107917}{300}{9}
      \Td{0.736068}{0.107917}{0}{9}
      \Al{0.746711}{0.126351}{300}{8}
      \C{0.757354}{0.107917}{300}{8}
      \Ar{0.742646}{0.096524}{300}{8}
      \Tl{0.763932}{0.096524}{180}{9}
      \Tl{0.763932}{0.096524}{240}{9}
      \Td{0.774575}{0.078089}{0}{9}
      \Td{0.774575}{0.078089}{60}{9}
      \Td{0.774575}{0.078089}{120}{9}
      \Al{0.753289}{0.078089}{60}{8}
      \Al{0.763932}{0.096524}{0}{8}
      \C{0.785218}{0.096524}{0}{8}
      \Ar{0.787731}{0.078089}{0}{8}
      \Tl{0.798374}{0.096524}{240}{9}
      \Tl{0.798374}{0.096524}{300}{9}
      \Tl{0.781153}{0.126351}{240}{7}
      \Tl{0.781153}{0.126351}{300}{7}
      \Tl{0.781153}{0.126351}{0}{7}
      \Td{0.819660}{0.096524}{60}{9}
      \Td{0.819660}{0.096524}{120}{9}
      \Td{0.819660}{0.096524}{180}{9}
      \Al{0.809017}{0.078089}{120}{8}
      \Al{0.798374}{0.096524}{60}{8}
      \C{0.809017}{0.114958}{60}{8}
      \Ar{0.826238}{0.107917}{60}{8}
      \Tl{0.815595}{0.126351}{300}{9}
      \Tl{0.815595}{0.126351}{0}{9}
      \Td{0.809017}{0.078089}{300}{7}
      \Td{0.809017}{0.078089}{0}{7}
      \Td{0.826238}{0.107917}{300}{9}
      \Td{0.826238}{0.107917}{0}{9}
      \Al{0.836881}{0.126351}{300}{8}
      \C{0.847524}{0.107917}{300}{8}
      \Ar{0.832816}{0.096524}{300}{8}
      \Tl{0.854102}{0.096524}{180}{9}
      \Tl{0.854102}{0.096524}{240}{9}
      \Tl{0.854102}{0.096524}{300}{9}
      \Ar{0.843459}{0.078089}{0}{8}
      \Td{0.843459}{0.078089}{240}{9}
      \Td{0.843459}{0.078089}{300}{9}
      \Al{0.864745}{0.078089}{240}{8}
      \C{0.854102}{0.059655}{240}{8}
      \Ar{0.836881}{0.066696}{240}{8}
      \Tl{0.847524}{0.048262}{120}{9}
      \Tl{0.847524}{0.048262}{180}{9}
      \Td{0.836881}{0.029828}{300}{9}
      \Td{0.836881}{0.029828}{0}{9}
      \Td{0.836881}{0.029828}{60}{9}
      \Al{0.826238}{0.048262}{0}{8}
      \Al{0.847524}{0.048262}{300}{8}
      \C{0.858167}{0.029828}{300}{8}
      \Ar{0.843459}{0.018434}{300}{8}
      \Tl{0.864745}{0.018434}{180}{9}
      \Tl{0.864745}{0.018434}{240}{9}
      \Tl{0.881966}{0.048262}{180}{7}
      \Tl{0.881966}{0.048262}{240}{7}
      \Td{0.909830}{-0.000000}{0}{7}
      \Td{0.909830}{-0.000000}{60}{7}
      \Td{0.909830}{-0.000000}{120}{7}
      \Td{0.875388}{-0.000000}{60}{9}
      \Td{0.875388}{-0.000000}{120}{9}
      \Al{0.854102}{-0.000000}{60}{8}
      \C{0.864745}{0.018434}{60}{8}
      \Ar{0.881966}{0.011393}{60}{8}
      \Tl{0.871323}{0.029828}{300}{9}
      \Tl{0.871323}{0.029828}{0}{9}
      \Tl{0.871323}{0.029828}{60}{9}
      \Ar{0.892609}{0.029828}{120}{8}
      \Td{0.892609}{0.029828}{0}{9}
      \Td{0.892609}{0.029828}{60}{9}
      \Al{0.881966}{0.048262}{0}{8}
      \C{0.903252}{0.048262}{0}{8}
      \Ar{0.905765}{0.029828}{0}{8}
      \Tl{0.916408}{0.048262}{240}{9}
      \Tl{0.916408}{0.048262}{300}{9}
      \Tl{0.916408}{0.048262}{0}{9}
      \Ar{0.927051}{0.029828}{60}{8}
      \Td{0.927051}{0.029828}{300}{9}
      \Td{0.927051}{0.029828}{0}{9}
      \Al{0.937694}{0.048262}{300}{8}
      \C{0.948337}{0.029828}{300}{8}
      \Ar{0.933629}{0.018434}{300}{8}
      \Tl{0.954915}{0.018434}{180}{9}
      \Tl{0.954915}{0.018434}{240}{9}
      \Td{0.965558}{-0.000000}{0}{9}
      \Td{0.965558}{-0.000000}{60}{9}
      \Td{0.965558}{-0.000000}{120}{9}
      \Al{0.944272}{-0.000000}{60}{8}
      \Al{0.954915}{0.018434}{0}{8}
      \C{0.976201}{0.018434}{0}{8}
      \Ar{0.978714}{-0.000000}{0}{8}
      \Tl{0.989357}{0.018434}{240}{9}
      \Tl{0.989357}{0.018434}{300}{9}
      \Tl{0.972136}{0.048262}{240}{7}
      \Tl{0.972136}{0.048262}{300}{7}
      \Tl{0.927051}{0.126351}{240}{5}
      \Tl{0.927051}{0.126351}{300}{5}
      \Tl{0.809017}{0.330792}{240}{3}
      \Tl{0.809017}{0.330792}{300}{3}
    \end{scope}
  \end{tikzpicture}
  \]
(c)
  \[
  \fractal
  \begin{tikzpicture}[scale=3]
    \begin{scope}[scale=.618034^2]
        \draw[ultra thick,fractalblue,->] (0,0) node[black,left] {$0$} 
                                        -- (2.618034,0) node[black,right] {$\phi^2$};
    \end{scope}
    \path (1.40,0.165396) node[font=\Large] {$\mapsto$};
    \begin{scope}[xshift=1.75cm,scale=.618034^2]
      \draw[ultra thick,fractalblue,<-] (0,0)  node[black,left] {$0$} 
                                    -- (.5, .86602540378443860)  node[black,left] {$\xi$} ;
       \draw[ultra thick,fractalblue,<-] (.5, .86602540378443860) 
                                   -- (1.5,.86602540378443860) node[black,right] {$\xi+1$} ;
      \draw[ultra thick,fractalred]  (1.5,.86602540378443860) 
                                  -- (1.618034,0) node[black,left] {$\phi$};
      \draw[ultra thick,fractalblue,->] (1.618034,0) -- (2.618034,0) node[black,right] {$\phi^2$};
    \end{scope}
    \begin{scope}[yshift=-0.6cm]
      \begin{scope}[xshift=-0.095492cm,scale=.618034^2]
        \draw[ultra thick,fractalred]  (.5, .86602540378443860)  node[black,left] {$\xi$}
                                    -- (2.618034,0) node[black,right] {$\phi^2$};
      \end{scope}
      \path (1.40,0.165396) node[font=\Large] {$\mapsto$};
      \begin{scope}[xshift=1.75cm,xshift=-0.095492cm,scale=.618034^2]
        \draw[ultra thick,fractalblue,<-] (.5, .86602540378443860)  node[black,left] {$\xi$} 
                                      -- (1.5,.86602540378443860) node[black,right] {$\xi+1$} ;
        \draw[ultra thick,fractalred]  (1.5,.86602540378443860) 
                                    -- (1.618034,0) node[black,left] {$\phi$};
        \draw[ultra thick,fractalblue,->] (1.618034,0) 
                                    -- (2.618034,0) node[black,right] {$\phi^2$};
      \end{scope}
    \end{scope}
  \end{tikzpicture}
  \]
  \caption{(a) The substitution rules for the hat fractal. All angles
    are multiples of $60^{\circ}$ and all ratios are golden.
    (b) The first few substitution steps.
    (c) The hat fractal can also be defined from a substitution on line segments.}
  \label{fig:fractal-rules}
\end{figure}

The hat fractal has several remarkable self-similarities and
symmetries. Most relevant for our purposes are the $120^{\circ}$
rotational symmetry of the curve obtained from the blue segment and
the $180^{\circ}$ rotational symmetry of the curve obtained from the
red segment (see Figure~\ref{fig:fractal-symmetries}).

\begin{figure}
  \[
  \begin{tikzpicture}[line join=round]
    \draw \thefractalbig;
    \draw[shift={(2.618033,0)},rotate=120] \thefractalbig;
    \draw[shift={(0.5*2.618033,0.866025*2.618033)},rotate=240] \thefractalbig;
    \draw (0,0) -- (2.618033,0) -- (0.5*2.618033,0.866025*2.618033) -- cycle;
  \end{tikzpicture}
  \qquad
  \begin{tikzpicture}[scale=1,line join=round]
    \draw \thefractalbigB;
    \draw (0.5,0.866025) -- (1,0) -- (2.618033,0) -- (2.118033,0.866025) -- cycle;
  \end{tikzpicture}
  \]
  \caption{Some symmetries of the hat fractal. The curve generated
    from the blue segment admits a $120^{\circ}$ symmetry, and that of
    the red segment has a $180^{\circ}$ symmetry.}
  \label{fig:fractal-symmetries}
\end{figure}
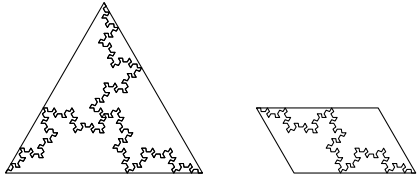

\subsection{The partition $\Plimit$ is new, but the hat fractal curve is not}

The hat fractal curve described in the previous section 
has appeared soon after the discovery of hat tilings \cite{MR4770585}.
In particular, it has been observed in the description of a partition
of the internal space associated with hat tilings \cite{zbMATH08135691}.
In fact, Figure~8 of \cite{zbMATH08135691} 
and Figure~10 from \cite{zbMATH08105556} bear an uncanny resemblance to our
Figure~\ref{fig:internal-space}, although it is not quite the same.
More precisely, the number of colors is not the same and, more importantly, the boundaries
of the partitions are not the same.
Figure~8 from \cite{zbMATH08135691} has 4 colors and was obtained by
considering the meta-tiles $T$, $H$, $P$ and $F$
following the original terminology of \cite{MR4770585}.
Since these meta-tiles are obtained by a desubstitution process,
it is possible that their fractal partition
can be obtained as the result of a Rauzy induction on our partition
\cite{MR4347332} for some hexagonal shape subwindow. 
For aperiodic tilings by Jeandel-Rao Wang tiles,
it was shown in {\cite{MR4226493}} that the sequence of 2-dimensional substitutions
obtained by a sequence of meta-tiles corresponds to the sequence of substitutions
obtained from a Rauzy induction procedure \cite{MR4347332} 
over the partitions of the window of the internal space into polygonal regions.
When the partition is made of convex polygons, it is easier to program 
using convex geometry, but the setup works for non convex (or fractal) regions as well.
We leave this question open for now.

\begin{question}
    Describe the relation between our Figure~\ref{fig:internal-space} and
    Figure~8 of \cite{zbMATH08135691}.
\end{question}

The substitution rules defining the hat fractal are also very closely related to the
``Golden Hex substitution'' considered in \cite{zbMATH08141691}. The Golden Hex
substitution was used to describe a particular member of the hat family of
tilings. Here, we use it to define a partition of the internal space associated
with the hat tilings.

\begin{question}
  Describe the relationship between the contracting graph-directed
  function system defined in Figure~\ref{fig:fractal-rules} and the
  Golden Hex substitution of \cite{zbMATH08141691}.
\end{question}

Finally, notice the hat fractal curve also appears in \cite{arXiv:2306.06512}
and in messages posted on Mathstodon in June 2023 by Pieter Mostert
\cite{2023-mostert-mathstodon}.
In these posts, a partition is presented that allows a hat
tiling to be constructed by a two-step process that the author calls a ``bad print job''.
The process involves anchors positioned somewhere in the interior of the hat tiles.
This illustrates that the placement of the anchors is fundamental.

\subsection{Description of the partition}

\begin{figure}
\begin{center}
  \begin{tabular}{cc}
    \raisebox{1cm}{
      \begin{tikzpicture}[
          scale=1.8
        ]
        \begin{scope}
        \clip (0,0) -- ++ (\Lxa,\Lxb) -- ++ (\Lya,\Lyb) -- ++ (-\Lxa,-\Lxb) -- cycle;
        \begin{pgfonlayer}{triangles}
          \clip (0,0) -- ++ (\Lxa,\Lxb) -- ++ (\Lya,\Lyb) -- ++ (-\Lxa,-\Lxb) -- cycle;
        \end{pgfonlayer}
        \foreach \a in {0,1}
        \foreach \b in {0,1}{
          \begin{scope}[xshift=\a*3.118034 cm + \b*0.809017 cm,
              yshift=\a*0.866025 cm + \b*3.133309 cm,
            ]
            \tikzstyle{every node}=[font=\normalsize]
            \FundamentalDomainAroundOrigin
            \FundamentalDomainLabellingRhombus

        \end{scope}}
        \end{scope}
        \draw[very thick,dotted] (0,0) -- ++ (\Lxa,\Lxb) -- ++ (\Lya,\Lyb) -- ++ (-\Lxa,-\Lxb) -- cycle;
    \end{tikzpicture}}
    &
    \raisebox{0cm}{
    \begin{tikzpicture}[
        scale=1.35,
      ]
      \tikzstyle{every node}=[font=\normalsize]
      \FundamentalDomainAroundOrigin
      \FundamentalDomainLabellingStandard
    \end{tikzpicture}}
    \\[-8mm]
    (a) A parallelogram-shaped fundamental domain. &
    (b) A radially symmetric fundamental domain.\\[7mm]
    \begin{tikzpicture}[
        scale=1.8,
      ]
      \tikzstyle{every node}=[font=\normalsize]
      \FundamentalDomainRecyclingShape
      \FundamentalDomainRecyclingLabelling

      \EmptyTriangleUP{0}{0}{0}
      \EmptyTriangleDOWN{\Lya}{\Lyb}{0}
      
    \end{tikzpicture}
    &
    \begin{tikzpicture}[
        scale=1.8,
      ]
      \FundamentalDomainRecyclingShapeSkeleton
    \end{tikzpicture}
    \\[3mm]
    (c) \parbox[t]{0.45\textwidth}{\raggedright A fundamental domain
    resembling the recycling symbol \recycl. It also resembles the
    logo of the Canadian Mathematical Society.} &
    (d) The boundaries of (c).
  \end{tabular}
\end{center}
  \vspace{-1em} %
  \caption{Partition of the torus $\C/\Lambda$ illustrated on some 
  fundamental domains for the action of $\Lambda$ on the internal space $\C$.
  Note that each partition of the fundamental domain includes two white triangles. }
  \label{fig:fundamental}
\end{figure}
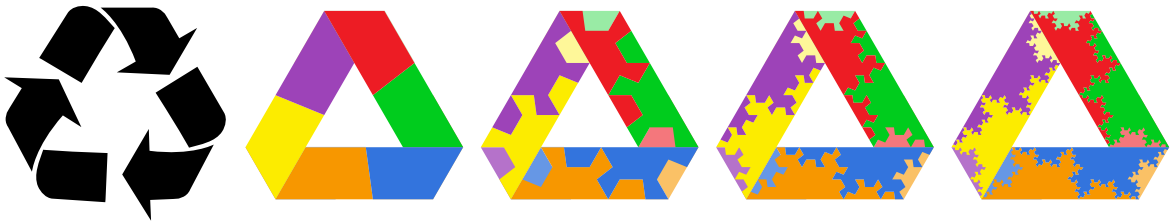
\begin{figure}
\begin{center}
    \setlength\tabcolsep{0mm}
    \begin{tabular}{lllll}
        {\resizebox{3cm}{!}{\recycl}}
        &
    \def\thefractalbig{\thefractalbigI}
    \def\thefractalsmall{\thefractalbigI}
    \begin{tikzpicture}[
        scale=0.8,
      ]
      \tikzstyle{every node}=[font=\footnotesize]
      \FundamentalDomainRecyclingShape
    \end{tikzpicture}
        &
    \def\thefractalbig{\thefractalbigII}
    \def\thefractalsmall{\thefractalbigII}
    \begin{tikzpicture}[
        scale=0.8,
      ]
      \tikzstyle{every node}=[font=\footnotesize]
      \FundamentalDomainRecyclingShape
    \end{tikzpicture}
        &
    \def\thefractalbig{\thefractalbigIII}
    \def\thefractalsmall{\thefractalbigIII}
    \begin{tikzpicture}[
        scale=0.8,
      ]
      \tikzstyle{every node}=[font=\footnotesize]
      \FundamentalDomainRecyclingShape
    \end{tikzpicture}
        &
    \def\thefractalbig{\thefractalbigIV}
    \def\thefractalsmall{\thefractalbigIV}
    \begin{tikzpicture}[
        scale=0.8,
      ]
      \tikzstyle{every node}=[font=\footnotesize]
      \FundamentalDomainRecyclingShape
    \end{tikzpicture}
    \end{tabular}
\end{center}
  \vspace{-1em} %
    \caption{A sequence of partitions converging to the fractal partition 
    $\Plimit$ of the internal space.}
\end{figure}

Figure~\ref{fig:internal-space} shows a partition 
$\Plimit=\{p_{-6},\dots,p_{-1},p_0,p_{1},\dots,p_6\}$
of the infinite
plane into 13 colors (the 12 tile colors plus the color white labelled 0). 
It is denoted with a plus $(+)$ sign because it describes only tilings where the
positively oriented hat is more frequent than the negatively oriented one. The other partition 
$\PlimitNEG$ can be obtained as its mirror image, while also swapping the signs of 
the labels.

We now describe this partition in more detail. It is obtained by the infinite
repetition of a fundamental domain. There are many possible ways of
choosing the fundamental domain, and three possibilities are shown in
Figure~\ref{fig:fundamental}. We will first focus on the fundamental
domain of Figure~\ref{fig:fundamental}(c), which is shaped like the
recycling symbol \recycl.
It consists of three parallelograms using
complementary colors (red-green, yellow-purple, and blue-orange), as
well as two white triangles. The acute angle of each parallelogram is
$60^{\circ}$. The sides of the parallelograms have lengths $\phi^2$
and $1$, respectively, where $\phi=\frac{1+\sqrt{5}}{2}$ is the golden
ratio. The white triangles have side length $\phi$. Each of the
parallelograms is divided into 4 colored regions using segments of the
hat fractal, as shown in Figure~\ref{fig:fundamental}(c).

We can also consider the parallelogram-shaped fundamental domain shown
in Figure~\ref{fig:fundamental}(a). Its two sides, in complex
coordinates, are $u=\phi^2+\xi$ and $v=\xi(\phi^2+\xi)$.
These two vectors are also the offset vectors by which the fundamental
domain repeats (this is irrespective of which presentation of the
fundamental domain is chosen). Their linear combinations with integer coefficients 
generate a lattice, which we denote 
\begin{equation}\label{eq:defLambda}
\Lambda
\,=\,\langle u,v\rangle_\Z
\,=\,\langle\phi^2+\xi,\xi(\phi^2+\xi)\rangle_\Z
\,\subseteq\, \Z[\phi,\xi]
\,\subseteq\, \C.
\end{equation}
We note that the area of the fundamental domain under translation by
this lattice is $\im u\bar v = 2\phi^2\sqrt 3$, whereas the area of
the two white triangles is $\frac{1}{2}\phi^2\sqrt 3$. Thus, the white
triangles make up exactly one quarter of the area of the fundamental
region. We also note that the long diagonal of the
parallelogram-shaped fundamental domain is
$u+v \approx 3.92705 + 3.99933i$. Its vertical offset is very close to
$4$, but not equal to it. Also, the direction of this diagonal appears
at first sight to be $45^{\circ}$, but this is not quite the case.

\subsection{Validity of the tiling}

We now prove that the partition generates valid tilings.
For every generic starting point $x\in\C/\Lambda$ in the torus,
the encoding of the shifted lattice $x+\Z[\xi]$
by the partition $\Plimit$ uniquely defines
a configuration 
\[
    w_x:\Z[\xi]\to\{0,\pm1,\pm2,\pm3,\pm4,\pm5,\pm6\}.
\]
First, we describe the possible patterns at support $\{z,z+1\}$ appearing in
the configuration $w_x$. These can be computed by listing the pair of
regions of the partition that intersect after translating one by one unit
horizontally.

\begin{lemma}\label{lem:language-at-z-and-z+1}
    Let $\Plimit=\{p_{-6},\dots,p_{-1},p_0,p_{1},\dots,p_6\}$
    be the partition shown in Figure~\ref{fig:internal-space}.
    For every $i,j\in\{0,\pm1,\pm2,\pm3,\pm4,\pm5,\pm6\}$,
    the intersection
    $(p_i+1)\cap p_j\neq\varnothing$
    is non-empty
    if and only if
    \begin{align*}
    (i,j)\in
        &\left\{
    \begin{array}{l}
        (+1, +6),
        (+2, +2),
        (+2, +5),
        (+2, +6),
        (+3, +1),
        (+3, +2),
        (+3, +4),
        (+3, +5),\\
        (+4, +2),
        (+4, +3),
        (+4, +6),
        (+5, +1),
        (+5, +5),
        (+5, +6),
        (+6, +1),\\
        (+1, -6),
        (+4, -5),
        (+5, -2),
        (+5, -3),
        (+5, -4),
        \\
            (-1, +2),
            (-2, +1),
            (-3, +4),
            (-5, +2),
            (-6, +2),\\
            (0, 0),
            (0, +1),
            (0, +3),
            (0, +4),
            (0, +6),
            (0, -1),
            (+1, 0),
            (+3, 0),
            (+4, 0),
            (+6, 0),
            (-4, 0),
    \end{array}
\right\}.
    \end{align*}
\end{lemma}

\begin{proof}
    The proof is done by inspecting
Figure~\ref{fig:internal+1}, which shows the partition, together with
an outline of the fundamental region from
Figure~\ref{fig:fundamental}(d), shifted by 1. For example, the figure
shows that the red $+1$ region, shifted by 1, intersects only the regions
$+6$, $-6$, and the white region. 
In total, there are 36 pairs $(i,j)$ 
corresponding to intersecting pair of regions,
that is, such that $(p_i+1)\cap p_j\neq\varnothing$.
\end{proof}

Then, we can show that the configuration $w_x$ describes a valid tiling.

\begin{proposition}\label{prop:validity-of-generated-tilings}
  For every generic starting point, the configuration generated from
  the partition describes a valid tiling of the plane by the hat.
\end{proposition}

\begin{proof}
  We must show that no two tiles generated by the partition overlap,
  and that there are no gaps between the tiles. We first show the
  no-overlap property. Consider a tile anchored at $z\in\C$ and another
  potential tile anchored at $z+1$. We show by case distinction that these
  tiles do not overlap.

More precisely, 
given the tile placed at a generic position $z\in\C$,
we deduce from Lemma~\ref{lem:language-at-z-and-z+1}
what are the possible tiles at position $z+1$ in the language 
of the configurations described by the toral $\Z^2$-action coded by the partition. 
The possible patterns made of two tiles are listed in Figure~\ref{fig:overlap-pairs} 
(thus ignoring those using label 0).
We observe that all pairs of patterns listed in Lemma~\ref{lem:language-at-z-and-z+1}
are non-overlapping.

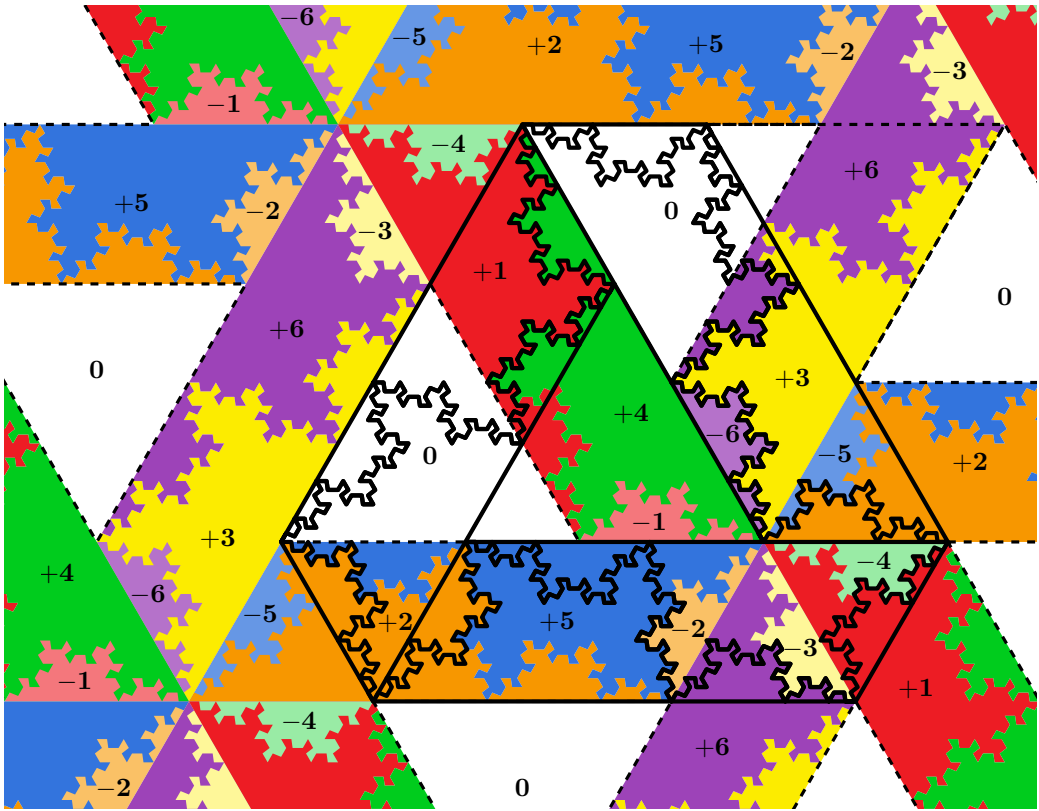
\begin{figure}
\begin{center}
\begin{tikzpicture}[
        scale=2.45
    ]
    \tikzstyle{every node}=[font=\normalsize]

    \clip (-1,-.58) rectangle (4.6, 3.78);
    \begin{pgfonlayer}{triangles}
      \clip (-1,-.58) rectangle (4.6, 3.78);
    \end{pgfonlayer}
    
    \foreach \a in {-1,0,1,2,3}
    \foreach \b in {-1,0,1,2}{
    \begin{scope}[xshift=\a*\Lxa cm + \b*\Lya cm,
                yshift=\a*\Lxb cm + \b*\Lyb cm]
                \FundamentalDomainAroundOrigin
    \end{scope}}

    \begin{scope}
      \clip (-0.8,-.52) rectangle (4.6, 3.77);
      \foreach \a in {-1,0,1,2,3}
      \foreach \b in {-1,0,1,2}{
        \begin{scope}[xshift=\a*\Lxa cm + \b*\Lya cm,
          yshift=\a*\Lxb cm + \b*\Lyb cm]
          \FundamentalDomainLabellingFigureTen
        \end{scope}}
    \end{scope}

    \begin{scope}[xshift=1cm,line width=1.6pt]
        \FundamentalDomainRecyclingShapeSkeleton
    \end{scope}
\end{tikzpicture}
\end{center}
\caption{Translating the fundamental region by the vector $1$}
\label{fig:internal+1}
\end{figure}
\begin{figure}
\def\pairwhitedots{
    \node[whitedot] at (0,0) {};
    \node[whitedot] at (1,0) {};
}
\def\scale{0.7}
\[
\begin{array}{c|l|l}
  \text{Tile at $z$} & \text{Tile at $z+1$} & \text{Tile pair pattern}\\
  \hline
    +1 & \{+6, -6\}&
    \begin{tikzpicture}[scale=\scale,baseline={(0,0)}] \smithsmall{0}{0}{+1}\smithsmall{1}{0}{+6}\pairwhitedots\end{tikzpicture}\quad
    \begin{tikzpicture}[scale=\scale,baseline={(0,0)}] \smithsmall{0}{0}{+1}\smithsmall{1}{0}{-6}\pairwhitedots\end{tikzpicture}
  \\
    +2 & \{+2, +5, +6\}&
    \begin{tikzpicture}[scale=\scale,baseline={(0,0)}] \smithsmall{0}{0}{+2}\smithsmall{1}{0}{+2}\pairwhitedots\end{tikzpicture}\quad
    \begin{tikzpicture}[scale=\scale,baseline={(0,0)}] \smithsmall{0}{0}{+2}\smithsmall{1}{0}{+5}\pairwhitedots\end{tikzpicture}\quad
    \begin{tikzpicture}[scale=\scale,baseline={(0,0)}] \smithsmall{0}{0}{+2}\smithsmall{1}{0}{+6}\pairwhitedots\end{tikzpicture}
  \\
    +3 & \{+1, +2, +4, +5\}&
    \begin{tikzpicture}[scale=\scale,baseline={(0,0)}] \smithsmall{0}{0}{+3}\smithsmall{1}{0}{+1}\pairwhitedots\end{tikzpicture}\quad
    \begin{tikzpicture}[scale=\scale,baseline={(0,0)}] \smithsmall{0}{0}{+3}\smithsmall{1}{0}{+2}\pairwhitedots\end{tikzpicture}\quad
    \begin{tikzpicture}[scale=\scale,baseline={(0,0)}] \smithsmall{0}{0}{+3}\smithsmall{1}{0}{+4}\pairwhitedots\end{tikzpicture}\quad
    \begin{tikzpicture}[scale=\scale,baseline={(0,0)}] \smithsmall{0}{0}{+3}\smithsmall{1}{0}{+5}\pairwhitedots\end{tikzpicture}
  \\
    +4 & \{+2, +3, +6, -5\} &
    \begin{tikzpicture}[scale=\scale,baseline={(0,0)}] \smithsmall{0}{0}{+4}\smithsmall{1}{0}{+2}\pairwhitedots\end{tikzpicture}\quad
    \begin{tikzpicture}[scale=\scale,baseline={(0,0)}] \smithsmall{0}{0}{+4}\smithsmall{1}{0}{+3}\pairwhitedots\end{tikzpicture}\quad
    \begin{tikzpicture}[scale=\scale,baseline={(0,0)}] \smithsmall{0}{0}{+4}\smithsmall{1}{0}{+6}\pairwhitedots\end{tikzpicture}\quad
    \begin{tikzpicture}[scale=\scale,baseline={(0,0)}] \smithsmall{0}{0}{+4}\smithsmall{1}{0}{-5}\pairwhitedots\end{tikzpicture}
  \\
    +5 & \begin{array}[t]{cc} \{+1, +5, +6,\\ -2, -3, -4\}\end{array} &
    \begin{tikzpicture}[scale=\scale,baseline={(0,0)}] \smithsmall{0}{0}{+5}\smithsmall{1}{0}{+1}\pairwhitedots\end{tikzpicture}\quad
    \begin{tikzpicture}[scale=\scale,baseline={(0,0)}] \smithsmall{0}{0}{+5}\smithsmall{1}{0}{+5}\pairwhitedots\end{tikzpicture}\quad
    \begin{tikzpicture}[scale=\scale,baseline={(0,0)}] \smithsmall{0}{0}{+5}\smithsmall{1}{0}{+6}\pairwhitedots\end{tikzpicture}\quad
    \begin{tikzpicture}[scale=\scale,baseline={(0,0)}] \smithsmall{0}{0}{+5}\smithsmall{1}{0}{-2}\pairwhitedots\end{tikzpicture}\quad
    \begin{tikzpicture}[scale=\scale,baseline={(0,0)}] \smithsmall{0}{0}{+5}\smithsmall{1}{0}{-3}\pairwhitedots\end{tikzpicture}\quad
    \begin{tikzpicture}[scale=\scale,baseline={(0,0)}] \smithsmall{0}{0}{+5}\smithsmall{1}{0}{-4}\pairwhitedots\end{tikzpicture}
  \\
    +6 & \{+1\}&
    \begin{tikzpicture}[scale=\scale,baseline={(0,0)}] \smithsmall{0}{0}{+6}\smithsmall{1}{0}{+1}\pairwhitedots\end{tikzpicture}
  \\
\end{array}
\]

\[
\begin{array}[t]{c|l|l}
    \text{Tile at $z$} & \text{Tile at $z+1$} & \text{Tile pair pattern}\\
    \hline
    &&\\[-2mm]
    -1 & \{+2\}&
    \begin{tikzpicture}[scale=\scale,baseline={(0,0)}] \smithsmall{0}{0}{-1}\smithsmall{1}{0}{+2}\pairwhitedots\end{tikzpicture}
        \\
    -2 & \{+1\}&
    \begin{tikzpicture}[scale=\scale,baseline={(0,0)}] \smithsmall{0}{0}{-2}\smithsmall{1}{0}{+1}\pairwhitedots\end{tikzpicture}
        \\
    -3 & \{+4\}&
    \begin{tikzpicture}[scale=\scale,baseline={(0,0)}] \smithsmall{0}{0}{-3}\smithsmall{1}{0}{+4}\pairwhitedots\end{tikzpicture}
        \\
\end{array}
\qquad
\begin{array}[t]{c|l|l}
    \text{Tile at $z$} & \text{Tile at $z+1$} & \text{Tile pair pattern}\\
    \hline
    &&\\[2mm]
    -4 & \varnothing               & \text{nothing to check}\\[2mm]
    -5 & \{+2\}&
    \begin{tikzpicture}[scale=\scale,baseline={(0,0)}] \smithsmall{0}{0}{-5}\smithsmall{1}{0}{+2}\pairwhitedots\end{tikzpicture}
        \\[-1mm]
    -6 & \{+2\}&
    \begin{tikzpicture}[scale=\scale,baseline={(0,0)}] \smithsmall{0}{0}{-6}\smithsmall{1}{0}{+2}\pairwhitedots\end{tikzpicture}
        \\
\end{array}
\]
\caption{Proof of non-overlap between tiles anchored at $z$ and
  $z+1$. We do not include anchor points labelled $0$, since they have
  no tile attached to them.}
\label{fig:overlap-pairs}
\end{figure}

Next, we show the analogous properties for anchors at offset
$1+\xi$. Consider a tile $A$ anchored at $z$ and a tile $B$ anchored
at $z+1+\xi$. An easy case distinction shows that the only overlapping
pairs of such tiles $(A,B)$ are $(+1,+4)$, $(+1,-6)$, $(+1,-5)$,
$(+6,-5)$, $(-2,-5)$, $(-2,+3)$, $(-2,+4)$, and $(-3,+4)$ (see
Figure~\ref{fig:overlaps}(a)).  It is easy to see from
Figure~\ref{fig:internal-space} that none of these overlapping pairs
occur in tilings generated from our partition: indeed, if we shift the
$+1$ region by $1+\xi$, it lies entirely outside of the $+4$, $-6$,
and $-5$ regions, and similarly for the other cases.

\begin{figure}
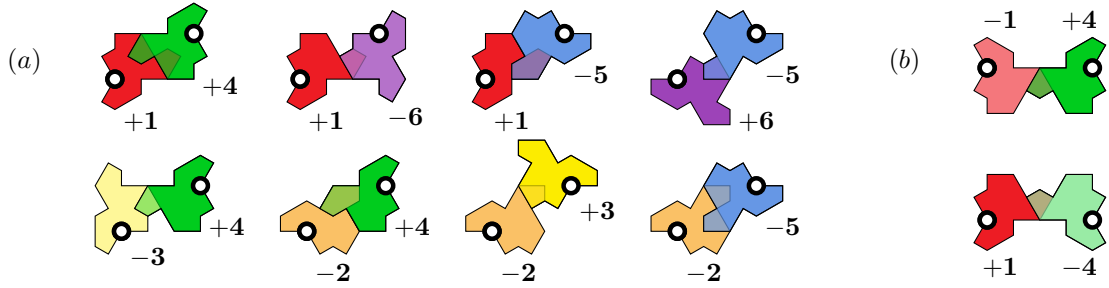

  \def\scale{0.7}
  \tikzstyle{every node}=[font=\relax]
  \[
  \begin{array}{ccccccc}
    (a)\quad
    & \overlapA{+1}{+4}{70}{(0.5,-0.8)}{(0.5,-1.0)}
    & \overlapA{+1}{-6}{80}{(0.5,-0.8)}{(0.5,-1.6)}
    & \overlapA{+1}{-5}{70}{(0.5,-0.8)}{(0.5,-0.8)}
    & \overlapA{+6}{-5}{70}{(1.5,-0.8)}{(0.5,-0.8)}
    & \qquad (b) \quad
    & \overlapB{-1}{+4}{60}{(0.25,0.9)}{(-0.25,0.9)}
    \\
    & \overlapA{-3}{+4}{40}{(0.5,-0.5)}{(0.5,-0.8)}
    & \overlapA{-2}{+4}{40}{(0.5,-0.8)}{(0.5,-0.8)}
    & \overlapA{-2}{+3}{70}{(0.5,-0.8)}{(0.5,-0.5)}
    & \overlapA{-2}{-5}{70}{(0.5,-0.8)}{(0.5,-0.8)}
    &
    & \overlapB{+1}{-4}{70}{(0.25,-0.9)}{(-0.25,-0.9)}
  \end{array}
  \]
  \caption{(a) Potential overlap between tiles anchored at $0$ and
    $1+\xi$. (b) Potential overlap between tiles anchored at $0$
    and $2$.}
  \label{fig:overlaps}
\end{figure}

  Next, we must show that there are no overlapping tiles at offset
  2. Suppose tile $A$ is anchored at $z$ and tile $B$ is anchored at
  $z+2$. If they overlap, their orientations must either be
  $(-1,+4)$ or $(+1,-4)$ (see Figure~\ref{fig:overlaps}(b)).
  Again, it is easy to see from Figure~\ref{fig:internal-space} that
  there are no such overlapping pairs: for example, the $-1$ region
  shifted by an offset of $2$ does not overlap the $+4$ region.

  By rotational symmetries, this shows that no two tiles overlap if
  their anchors are at distance 2 or less. Tiles whose anchors are
  more than 2 units apart cannot overlap at all; so there are no
  overlaps in a tiling generated by the partition.

  To show that there are no gaps between the tiles, assume there was
  such a gap. Consider the colors of the tiles surrounding the gap (up
  to some fixed but sufficiently large distance). The colors of these
  tiles are generated by the corresponding colorings of finitely many
  points $z_1,\ldots,z_n$ in the partition $\Plimit$. Since we assumed
  the grid to be in general position, none of these points lie on the
  partition boundaries. Therefore, there exists some $\epsilon>0$ such
  that the $\epsilon$-neighborhoods of $z_1,\ldots,z_n$ do not overlap
  any partition boundaries. This implies that there must be infinitely
  many places where there is a gap in the tiling, and these places
  must, in the limit, make up some positive fraction of the anchor
  points. Then some positive fraction of kites is not covered by a
  tile.

  On the other hand, we noted that the area of the white triangles in
  the partition is exactly 1/4 of the total area. This means that in
  our tiling, in the limit, exactly 3/4 of the grid points must anchor
  a tile. Since each tile covers 8 kites and each anchor point adjoins
  6 kites, this implies that each kite is, on average, covered by
  exactly 1 tile, contradicting the previous paragraph. Hence there are no
  gaps in the tiling.
\end{proof}

So far, we have assumed a ``generic'' starting point, i.e., such that
no grid points fall on partition boundaries. When the starting point
is not generic,
a point on a partition boundary may correspond to more than one valid
tiling.  A choice can be made as to which set of the partition it
belongs to.  Other points in its orbit may also fall on the boundary
where a consistent choice must be made. This is formalized in
Section~\ref{sec:Markov-partition} below as a symbolic dynamical system defined
from a topological partition of the space into open sets. 

The subset of points of an
orbit intersecting the partition boundary describes what is a called a
\emph{Conway worm}, and the choices of tiles are called its
\emph{resolutions} \cite{MR1355301}; see also \cite[10.5.8]{MR857454}
and \cite[Figure 7.22]{MR3136260}. For many tile sets, Conway worms are within
bounded distance of a straight line and can be described by their slope
\cite{MR4730985}.  But, in the context of the hat tilings, it is known
that they may have fractal (\text{snake}) or linear (Conway worm)
shape \cite{PhysRevB.108.224109},
see also Figure 49 in \cite{tatham_finite-state_2025}.
Their fractal vs.\@ linear nature can be explained by the fractal and linear
nature of the partition's sets, and by the subsets of orbits under the
$\Z^2$-action $\Rhat$ by horizontal and vertical unit translation on the torus
$\R^2/\Lambda$ that remain in its boundary.

\section{A Markov partition for the hat tilings}

The goal of this section is to show that the partition $\Plimit$ is in fact
a Markov partition for the dynamical system 
$\dynsys{\Z^2}{\Rhat}{\C/\Lambda}$ defined below in \eqref{eq:z2-action-hat}.

\subsection{Markov partition}\label{sec:Markov-partition}

Markov partitions were originally defined for one-dimensional dynamical
systems $\dynsys{\Z}{R}{\T^2}$ and were extended to $\Z^d$-actions by automorphisms of
compact Abelian groups in \cite{MR1632169}.
Following \cite{MR4213162,MR4347332}, we use the notion of Markov partition
proposed in \cite[\S 6.5]{MR1369092} for dynamical systems defined by 
a $\Z^2$-action on a torus.

Let $M$ be a compact metric space.
Consider $\dynsys{\Z^2}{R}{M}$, a continuous $\Z^2$-action on $M$ where
$R\colon \Z^2\times M\to M$.
For some finite set $\Acal$,
a \emph{topological partition} of $M$ is a
collection $\{P_a\}_{a\in\Acal}$ of disjoint open sets $P_a\subset M$
such that $M = \bigcup_{a\in\Acal} \overline{P_a}$. If $S\subset\Z^2$ is a finite set,
we say that a pattern $w\in\Acal^S$
is \emph{allowed} for $\Pcal,R$ if
\begin{equation}\label{eq:allowed-if-nonempty}
	\bigcap_{\bk\in S} R^{-\bk}(P_{w_\bk}) \neq \varnothing.
\end{equation}

Let us recall that a $\Z^2$-\emph{subshift} is a set of the form $X \subset \Acal^{\Z^2}$ which is closed in the prodiscrete topology and invariant under the shift action.
Let $\Lcal_{\Pcal,R}$ be the collection of all allowed patterns for $\Pcal,R$.
The set $\Lcal_{\Pcal,R}$ is the language of a subshift 
$\Xcal_{\Pcal,R}\subseteq\Acal^{\Z^2}$ 
\cite[Prop.~9.2.4]{MR3525488}.
The subshift $\Xcal_{\Pcal,R}$ is a \emph{symbolic dynamical system} defined
from the dynamical system $\dynsys{\Z^2}{R}{M}$ and the partition $\Pcal$.

For every configuration $x\in\Xcal_{\Pcal,R}$ and $m\geq 0$, there is a
corresponding nonempty open set
\[
D_m(x) = \bigcap_{\Vert\bk\Vert_{\infty} \leq m} R^{-\bk}(P_{x_\bk}) \subset M.
\]
The sequence of compact closures $(\overline{D}_m(x))_{m \in \N}$ of these sets is nested and thus it follows that their intersection is nonempty. Notice that there is no reason why $\operatorname{diam}(\overline{D}_m(x))$ should converge to zero, and thus the intersection could contain more than one point. In order for $\Xcal_{\Pcal,R}$ to capture the dynamics of $\dynsys{\Z^2}{R}{M}$, this intersection should contain only one point.
This leads to the following definition.

\begin{definition}
    [\cite{MR1369092}]
    \label{def:toppartition}
	A topological partition $\Pcal$ of $M$ gives a \emph{symbolic representation} $\Xcal_{\Pcal,R}$ of $\dynsys{\Z^2}{R}{M}$
	if for every $x\in\Xcal_{\Pcal,R}$ the intersection
	$\bigcap_{m=0}^{\infty}\overline{D}_m(x)$ consists of exactly one
	point $\rho \in M$.
	We call $x$ a \emph{symbolic representation of $\rho$}. 
\end{definition}

\begin{definition}
    [\cite{MR1369092,MR4213162,MR4347332}]
    \label{def:Markov}
A topological partition $\Pcal$ of $\T^2$ is a \emph{Markov partition} for
$\dynsys{\Z^2}{R}{\T^2}$ if 
\begin{itemize}
    \item $\Pcal$ gives a symbolic representation of $\dynsys{\Z^2}{R}{\T^2}$ and 
    \item $\Xcal_{\Pcal,R}$ is a (2-dimensional) shift of finite type (SFT).
\end{itemize}
\end{definition}

\subsection{A Markov partition for the hat tilings}

The proof of existence of a Markov partition for Jeandel-Rao tilings
was the culmination of a work split into three articles 
\cite{MR4226493,MR4347332,MR4213162}.
Here, the proof is simpler because we can use results already known 
for the hat tilings.
Recall that, in any tiling by the hat, the ratio of frequencies of the two
orientations of the hat is $[1:\phi^4]$ or $[\phi^4:1]$ \cite{MR4770585}.
More precisely, the set of tilings by the hat splits into two disjoint components
\[
\Omega_\text{hat}
=
\Omega_\text{hat}^+
\cup
\Omega_\text{hat}^-,
\]
where $\Omega_\text{hat}^+$ is the component where positive orientations are more frequent,
that is, tilings with more hats than anti-hats.
It is known that each component $\Omega_\text{hat}^+$ and $\Omega_\text{hat}^-$
defines a minimal dynamical system for the shift action \cite{MR4770585,zbMATH08135691}.
Recall that a subshift $X\subseteq\Acal^{\Z^2}$ is \emph{minimal} if 
it does not contain any proper nonempty subshift,
that is, $Y\subseteq X$ implies $Y=X$ for every nonempty subshift $Y$.

On the torus $\C/\Lambda$ where $\Lambda$ is the lattice defined in
\eqref{eq:defLambda}, we consider the following $\Z^2$-action $\Rhat$
acting by horizontal and vertical unit steps:
\begin{equation}\label{eq:z2-action-hat}
\begin{array}{rccl}
    \Rhat:&\Z^2\times\C/\Lambda & \to & \C/\Lambda\\
    &(\bk,\bx) & \mapsto &\bx+\bk.
\end{array}
\end{equation}
The $\Z^2$-action $\Rhat$ defines a dynamical system $\dynsys{\Z^2}{\Rhat}{\C/\Lambda}$
and we have the following theorem about the symbolic dynamical system $\Xcal_{\Plimit,\Rhat}$
determined by the partition $\Plimit$.

\begin{theorem}
    \label{thm:equality-holds-we-describe-all-tilings}
    We have $\Xcal_{\Plimit,\Rhat}=\Omega_\text{hat}^+$.
\end{theorem}

\begin{proof}
    From Proposition~\ref{prop:validity-of-generated-tilings},
    the configurations
    $\Z[\xi]\to\{0,\pm1,\pm2,\pm3,\pm4,\pm5,\pm6\}$
    in the set $\Xcal_{\Plimit,\Rhat}$ describe valid tilings.
    Thus, $\Xcal_{\Plimit,\Rhat}\subseteq\Omega_\text{hat}^+$.
    It is known that $\Omega_\text{hat}^+$ is a minimal dynamical system 
    \cite{MR4770585,zbMATH08135691}.
    We have that $\Xcal_{\Plimit,\Rhat}\neq\varnothing$ is a nonempty subshift.
    Therefore, from the minimality of $\Omega_\text{hat}^+$,
    we deduce that the equality $\Xcal_{\Plimit,\Rhat}=\Omega_\text{hat}^+$ holds.
\end{proof}

In lay terms, the above result can be summarized as saying that not
only does our partition generate valid hat tilings, but
conversely, all valid hat tilings with more hats than anti-hats 
can be generated by our partition (subject to the above provisos about
disambiguating points that fall on the partition boundaries).

\begin{theorem}
    \label{thm:is-a-markov-partition}
    $\Plimit$ is a Markov partition for the dynamical system 
    $\dynsys{\Z^2}{\Rhat}{\C/\Lambda}$.
\end{theorem}

\begin{proof}
    The fact that the partition $\Plimit$ gives a symbolic representation of 
    $\dynsys{\Z^2}{\Rhat}{\C/\Lambda}$ follows from
    \cite[Lemma~3.4]{MR4213162} because
    $\dynsys{\Z^2}{\Rhat}{\C/\Lambda}$ is a minimal dynamical system, and
    there exists an atom of the partition $\Plimit$ which is
    invariant only under the trivial translation in $\C/\Lambda$.

    By definition, $\Omega_\text{hat}$ is a shift of finite type.
    The subset $\Omega_\text{hat}^+$ is also a shift of finite type because
    it can be described by additional forbidden patterns forcing the hat with
    negative orientations not to be adjacent.
    From Theorem~\ref{thm:equality-holds-we-describe-all-tilings},
    we conclude that the $2$-dimensional subshift $\Xcal_{\Plimit,\Rhat}$ 
    is a shift of finite type. Thus, $\Plimit$ is a Markov partition for the
    dynamical system $\dynsys{\Z^2}{\Rhat}{\C/\Lambda}$.
\end{proof}

\subsection{Final remarks}

\begin{remark}
    As a consequence of Theorem~\ref{thm:equality-holds-we-describe-all-tilings}
    and Lemma~\ref{lem:language-at-z-and-z+1},
    the following hat-hat patterns
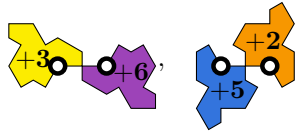
\[
\begin{tikzpicture}[scale=.65,baseline={(0,0)}] \smithsmall{0}{0}{+3}\smithsmall{1}{0}{+6} \end{tikzpicture},\quad
\begin{tikzpicture}[scale=.65,baseline={(0,0)}] \smithsmall{0}{0}{+5}\smithsmall{1}{0}{+2} \end{tikzpicture}
\]
and the following hat-antihat patterns
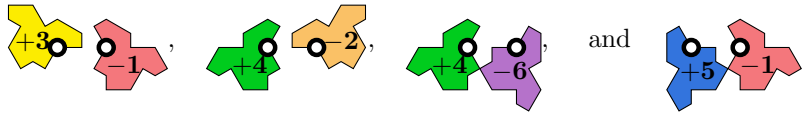
\[
\begin{tikzpicture}[scale=.65,baseline={(0,0)}] 
\smithsmall{0}{0}{+3}\smithsmall{1}{0}{-1}
\end{tikzpicture},\quad
\begin{tikzpicture}[scale=.65,baseline={(0,0)}] 
\smithsmall{0}{0}{+4}\smithsmall{1}{0}{-2}
\end{tikzpicture},\quad
\begin{tikzpicture}[scale=.65,baseline={(0,0)}] 
\smithsmall{0}{0}{+4}\smithsmall{1}{0}{-6}
\end{tikzpicture},\quad\text{ and }\quad
\begin{tikzpicture}[scale=.65,baseline={(0,0)}] 
\smithsmall{0}{0}{+5}\smithsmall{1}{0}{-1}
\end{tikzpicture},
\]
do not belong to the language of 
$\Omega_\text{hat} =\Omega_\text{hat}^+ \cup \Omega_\text{hat}^-$.
It is an interesting exercise to show that they
do not admit arbitrarily large surroundings by isometric copies of the hat.
Moreover, we may deduce from the same two results 
that a configuration $w\in\Omega_\text{hat}$ belongs to
the positive component $\Omega_\text{hat}^+$ if and only if it contains one
of the following patterns:
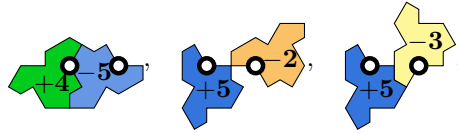
\[
    \begin{tikzpicture}[scale=.65,baseline={(0,0)}] \smithsmall{0}{0}{+4}\smithsmall{1}{0}{-5}\end{tikzpicture},\quad
    \begin{tikzpicture}[scale=.65,baseline={(0,0)}] \smithsmall{0}{0}{+5}\smithsmall{1}{0}{-2}\end{tikzpicture},\quad
    \begin{tikzpicture}[scale=.65,baseline={(0,0)}] \smithsmall{0}{0}{+5}\smithsmall{1}{0}{-3}\end{tikzpicture}.
\]
\end{remark}

The hat shape is only one member of the family of aperiodic shape described in
\cite{MR4770585}. Thus, many questions remain open.

\begin{question}
    \label{question:markov-for-hat-turtle-family}
    Find Markov partitions, if they exist, describing other members of the hat
    family of aperiodic tilings described in \cite{MR4770585} including tilings
    by the turtle shape. 
\end{question}

We note that after an initial draft of this paper was circulated,
Pieter Mostert quickly found such a partition for the turtle tilings, for
``turtle in hat'' tilings and for ``hat in turtle'' tilings.

\begin{question}
    \label{question:markov-for-spectre-caspr}
    Find Markov partitions, if they exist, describing tilings by the Spectre
    \cite{MR4807152} and the CASPr tilings
    \cite{zbMATH08105556,baake_long-range_2025}.
\end{question}

The partition $\Plimit$ is the first example of a Markov partition for a
$\Z^2$-action on a torus having fractal boundaries. Previous Markov partitions
for $\Z^2$-actions on a torus were all polygonal \cite{MR4213162,MR4963140}.
We believe that the fractal nature of the partition comes from the interaction
of two distinct quadratic number fields $\Q(\sqrt{5})$ and $\Q(\sqrt{3})$
in the description of the substitutive structure of hat tilings.
In general, the characterization of Markov partitions for $\Z^2$-action on a
torus and their fractal vs.\@ smooth (linear) boundaries remains an open question.

\begin{question}
    Characterize the nature (fractal or smooth) of the boundary of the Markov
    partition of a toral $\Z^2$-action. In particular, when the boundary is smooth,
    is it always piecewise linear as for hyperbolic toral automorphisms
    \cite{MR1145614}?
\end{question}

    The partition $\Plimit$ is a partition of the internal space
    of a cut and project scheme \cite{MR3136260}
    which is \emph{degenerate} \cite[Section 7]{labbe:tel-05138330}.
    Degenerate cut and project schemes 
    do not obey the standard hypothesis
    of injectivity of the projection of some higher dimensional lattice
    to the physical space.
    Their introduction was motivated by the description 
    of occurrences of patterns in Jeandel-Rao aperiodic tilings \cite{MR4213162}.
    As for hat tilings whose tiles are placed on the lattice $\Z[\xi]$, 
    Jeandel-Rao tilings are made of unit square Wang tiles placed on the periodic lattice $\Z^2$.
    To keep this article short and accessible, we do not include the description of the
    degenerate cut and project scheme here.

\section*{Acknowledgements}

The authors are thankful to Michael Baake 
for helpful discussions about the minimality of $\Omega_\text{hat}^+$ 
and Craig Kaplan for comments on an earlier draft.

The authors acknowledge support of the Institut Henri Poincaré (UAR 839
CNRS-Sorbonne Université), and LabEx CARMIN (ANR-10-LABX-59-01).
This work was presented for the first time at Institut Henri Poincaré (Paris,
March 26th, 2026) during the semester \emph{Illustration as a mathematical
research technique}.

The first author acknowledges Université de Bordeaux's program
``\textit{Mobilité internationale des personnels de recherche}''
partially supporting a one-year stay at CRM-CNRS in Montréal (2025-2026).

This work was partly funded from France's Agence Nationale de la Recherche
(ANR) project IZES (ANR-22-CE40-0011).
It was also supported by grants from the
\emph{Symbolic Dynamics and Arithmetic Expansions}
(SymDynAr) Project, co-funded by ANR (ANR-23-CE40-0024)
and FWF (\href{https://dx.doi.org/10.55776/I6750}{I 6750}),
the Austrian Science Fund, as well as the Natural Sciences and Engineering
Research Council of Canada (NSERC).

\bibliographystyle{abbrv}
\bibliography{biblio}

\end{document}